\documentclass[reqno,a4paper]{amsart}

\usepackage{mathrsfs,amsmath,amssymb,graphicx}
\usepackage{bbm}
\usepackage{stmaryrd}
\usepackage[unicode]{hyperref}
\usepackage{latexsym}
\usepackage{layout}
\usepackage[english]{babel}
\usepackage{color}
\usepackage[subrefformat=parens, labelfont=up]{subcaption}
\usepackage{fancyvrb}
\usepackage{color}
\usepackage{amsmath,amscd}
\usepackage{multicol}
\usepackage{multirow}
\usepackage{enumitem}
\usepackage[percent]{overpic}
\usepackage{mathtools}
\usepackage{tikz} 
\usepackage{transparent}

\usetikzlibrary{matrix,arrows,decorations,shapes,calc}
\tikzset{every picture/.append style={remember picture},
na/.style={baseline=-.5ex}}

\theoremstyle{plain}
\newtheorem{theorem}{Theorem}[section]
\newtheorem{lemma}[theorem]{Lemma}

\newtheorem{corollary}[theorem]{Corollary}

\newtheorem{question}[theorem]{Question}
\theoremstyle{definition}
\newtheorem{definition}{Definition}[section]
\theoremstyle{remark}
\newtheorem{remark}{Remark}[section]

\newcommand{\nada}[1]   {}
\newcommand{\cout}[1]   {}

\newcommand{\Sbb} {\mathbb S}

\newcommand{\sphere} {\Sbb^3}

\newcommand{\Ztwo} {\mathbb Z_2}
\newcommand{\Zone} {\mathbbm  1}
 
\newcommand{\sym}[1]{\mathrm{MCG}(\sphere,#1)} 
\newcommand{\mcg}{\mathrm{MCG} } 
\newcommand{\pmcg}{\mathrm{MCG}_+ } 
\newcommand{\psym}[1]{\mathrm{MCG}_+(\sphere,#1)}

\newcommand{\paut}{\mathrm{Homeo}_+}
\newcommand{\id}{\mathrm{id}}
\newcommand{\rel}{\mathrm{rel} }
 
\newcommand{\Compl}[1]{E(#1)}
\newcommand{\Split}[1]{\widehat{E}(#1)}
\newcommand{\rnbhd}[1]{N(#1)}
\newcommand{\openrnbhd}[1]{\mathring{N}(#1)}

\newcommand{\hcfourone}{\mathbf{h4}_1} 
\newcommand{\hctwoone}{\mathbf{h2}_1}

\newcommand{\fivetwo}{{\mathbf 5_{2}}} 
\newcommand{\sixthirteen}{{\mathbf 6_{13}}} 
\newcommand{\sixtwelve}{{\mathbf 6_{12}}} 
\newcommand{\seventhirtynine}{{\mathbf  7_{39}}}
\newcommand{\seventhirtyeight}{{\mathbf  7_{38}}}
\newcommand{\seventhirtysix}{{\mathbf  7_{36}}}
\newcommand{\seventhirtytwo}{{\mathbf  7_{32}}}
\newcommand{\sevenfiftynine}{{\mathbf 7_{59}}}
\newcommand{\sevensixty}{{\mathbf  7_{60}}}
 
\newcommand{\draftMP}[1]{\ifdraft{\color{blue}#1}\fi}
\newcommand{\draftGB}[1]{\ifdraft{\color{red}#1}\fi}
\newcommand{\draftYW}[1]{\ifdraft{\color{orange}#1}\fi}
\newcommand{\draftGP}[1]{\ifdraft{\color{teal}#1}\fi}

\numberwithin{equation}{section}
\newif\ifdraft
\draftfalse

\title[Tangle replacement]{Handlebody-knots obtained by tangle replacement}
\author[G.\ Bellettini]{Giovanni Bellettini}
\address{Dipartimento di Scienze Matematiche, Informatiche e Fisiche,
via delle Scienze 206, 33100 Udine UD, Italy
and International Centre for Theoretical Physics ICTP,
Mathematics Section, 34151 Trieste, Italy
}
\email{giovanni.bellettini@uniud.it}

\author[G.\ Paolini]{Giovanni Paolini}
\address{Dipartimento di Matematica, Universit\`a di Bologna, Piazza di Porta
San Donato 5, 40126 Bologna, Italy}
\email{g.paolini@unibo.it}

\author[M.\ Paolini]{Maurizio Paolini}
\address{Dipartimento di Matematica e Fisica, Universit\`a Cattolica del Sacro Cuore, 25121 Brescia, Italy}
\email{maurizio.paolini@unicatt.it}

\author[Y.\ S.\ Wang]{Yi-Sheng Wang}
\address{National Sun Yat-sen University, Kaohsiung, Taiwan}
\email{yisheng@math.nsysu.edu.tw}

\date{\today}

\begin{document}

\thanks{The fourth author gratefully acknowledges the support from NSTC, Taiwan (grant no. 112-2115-M-110 -001 -MY3).}

\begin{abstract}
We study tangle replacement in the context of handlebody-knots. The main results classify the handlebody-knots obtained by performing tangle replacement on the prime spatial handcuff graph with four crossings, and determine their symmetry. The work is motivated by the study of small crossing handlebody-knots that are difficult to distinguish with computational invariants and by their chirality problem.
\end{abstract}

\maketitle

%
%

\draftMP{DraftMP comments by Maurizio}
\draftYW{DraftYW comments by Yi-Sheng}
\draftGB{DraftGB comments by Giovanni B}
\draftGP{DraftGP comments by Giove}

\section{Introduction}\label{sec:intro}

A \emph{spatial graph} is an embedded finite connected CW-complex $\Gamma$ in the oriented $3$-sphere $\sphere$. The classical knot theory studies the case where $\Gamma$ is a cycle. 
A \emph{genus $g$ handlebody-knot} is an embedded genus $g$ handlebody in $\sphere$. 
Two spatial graphs $\Gamma,\Gamma'$ (resp.\ two handlebody-knots $V,V'$) are \emph{equivalent} 
if there is an orientation-preserving self-homeomorphism $f$ of $\sphere$ with
$f(\Gamma)=\Gamma'$ (resp.\ $f(V)=V'$).

A regular neighborhood of a spatial graph 
determines a unique handlebody-knot, up to equivalence. Conversely, every spine of a handlebody-knot is a spatial graph, and it is typical to represent a handlebody-knot by its spines. A handlebody-knot may, however, have infinitely many spines, inequivalent as spatial graphs. Since the spine of a genus one handlebody is unique, up to ambient isotopy, the study of genus one handlebody-knots is equivalent to the classical knot theory. A spatial graph is \emph{trivial} if it is contained in a $2$-sphere in $\sphere$, and a handlebody-knot is \emph{trivial} if it has a trivial spine. The present study concerns the following question:
\begin{question}\label{ques:informal}
How local changes to a spine of a handlebody-knot affects its topology?
\end{question}

Figs.\ \ref{fig:fivetwo}, \ref{fig:sixthirteen}, as well as Figs.\ \ref{fig:seventhirtysix}, \ref{fig:seventhirtyeight} show that local modifications of spines may change one handlebody-knot into another. On the other hand, both Figs.\ \ref{fig:trivial_I}, \ref{fig:trivial_II} represent the trivial handlebody-knot, and hence provide an example where local change of spines does not alter the topology of handlebody-knots. In the classical knot theory, many operations takes the form of tangle replacement, for instance, an unknotting operation or more generally a rational unknotting operation, which changes a knot drastically, or at the opposite end, the knot mutation, which produces knots hardly distinguishable; see Iltgen-Lewark-Marino \cite{IltLewMar:25}, Hom \cite{Hom:22} and references therein.     

\begin{figure}[t]
	\begin{subfigure}{.32\linewidth}
		\centering
		\begin{overpic}[scale=.11,percent]{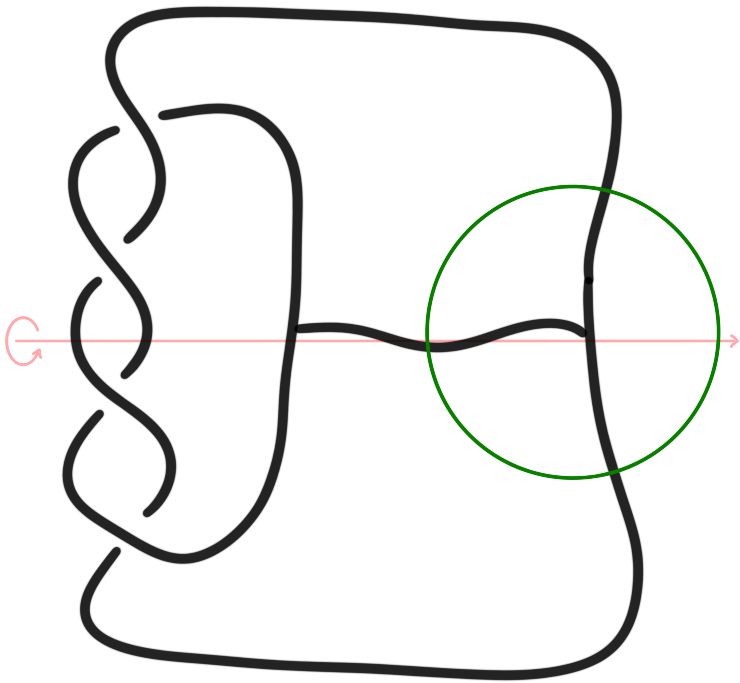}
			\put(0,35){\footnotesize $\pi$}
			\put(50,40){\footnotesize $x$}
			\put(84,68){\footnotesize $y$}
			\put(78,22){\footnotesize $z$}
		\end{overpic}
		\caption{$\hcfourone$.}
		\label{fig:hc_fourone}
	\end{subfigure}
	\begin{subfigure}{.32\linewidth}
		\centering
		\begin{overpic}[scale=.11,percent]{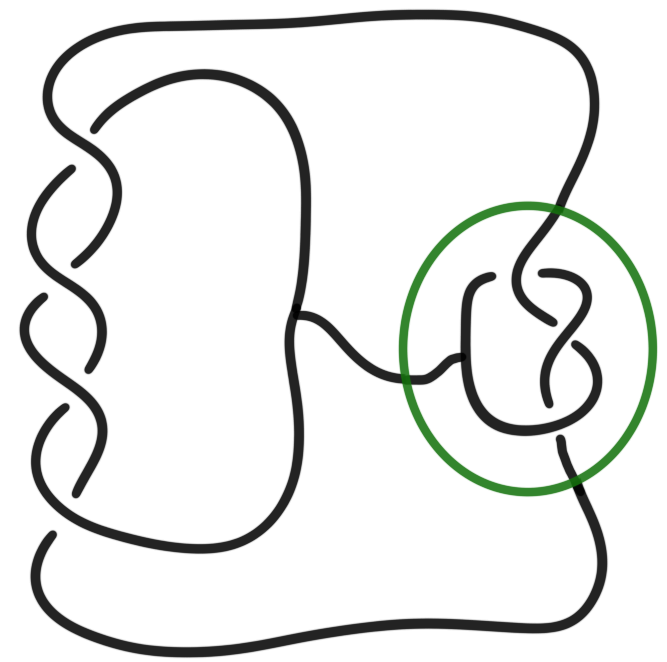}
			\put(68.5,58.5){\footnotesize $B$} 
		\end{overpic}
		\caption{$\fivetwo$.}
		\label{fig:fivetwo}
	\end{subfigure}
	\begin{subfigure}{.32\linewidth}
		\centering
		\begin{overpic}[scale=.11,percent]{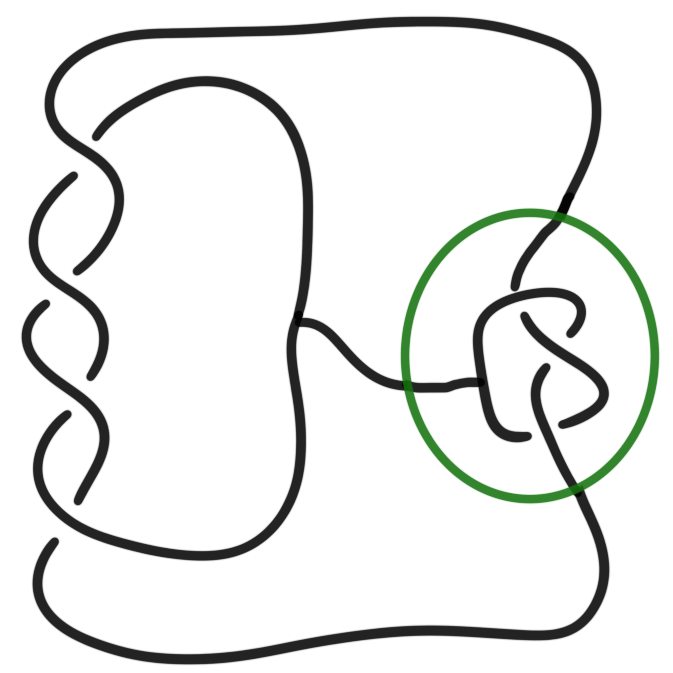}
			\put(67,57){\footnotesize $B$} 
		\end{overpic}
		\caption{$\sixthirteen$.}
		\label{fig:sixthirteen}
	\end{subfigure}
	\begin{subfigure}{.32\linewidth}
		\centering
		\begin{overpic}[scale=.1,percent]{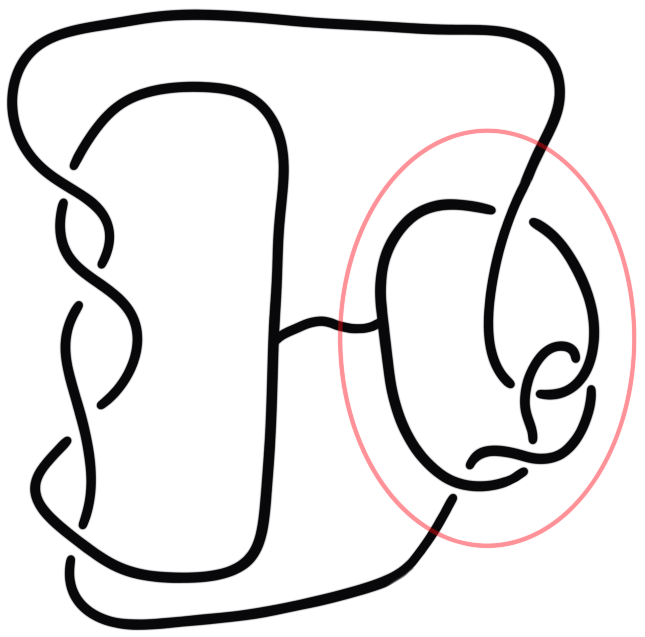}
		\end{overpic}
		\caption{$\seventhirtysix$.}
		\label{fig:seventhirtysix}
	\end{subfigure}
	\begin{subfigure}{.32\linewidth}
		\centering
		\begin{overpic}[scale=.1,percent]{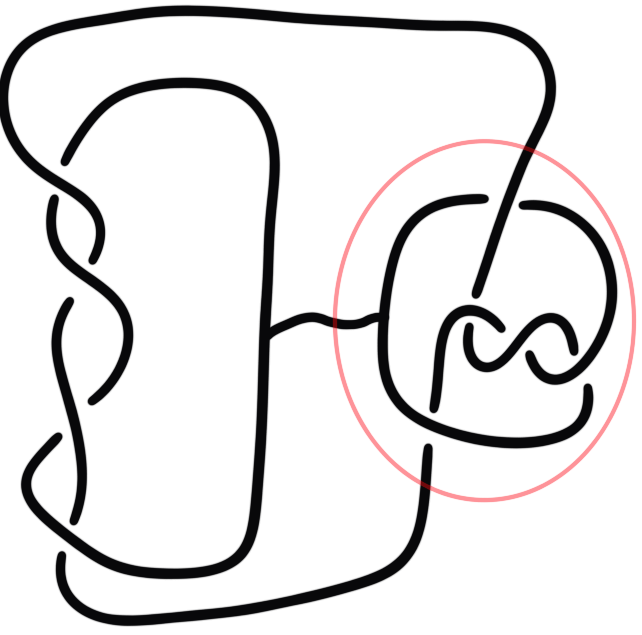}
		\end{overpic}
		\caption{$\seventhirtyeight$.}
		\label{fig:seventhirtyeight}
	\end{subfigure}
    \caption{\ref{fig:fivetwo}, \ref{fig:sixthirteen}, \ref{fig:seventhirtysix}, \ref{fig:seventhirtyeight} are handlebody-knots in the handlbody-knot tables \cite{IshKisMorSuz:12,BelPaoPaoWan:25p}. 
    }
	\label{fig:small_crossing}
\end{figure}

\begin{figure}[b]
	\begin{subfigure}{.24\linewidth}
		\centering
		\begin{overpic}[scale=.09,percent]{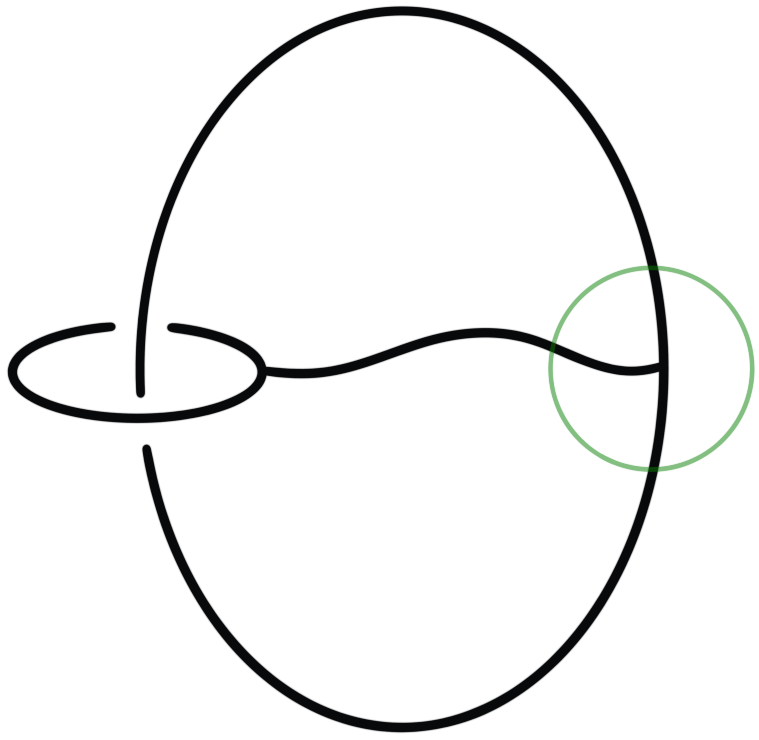}
		\end{overpic}
		\caption{$\hctwoone$.}
		\label{fig:hc_two}
	\end{subfigure}
	\begin{subfigure}{.24\linewidth}
		\centering
		\begin{overpic}[scale=.09,percent]{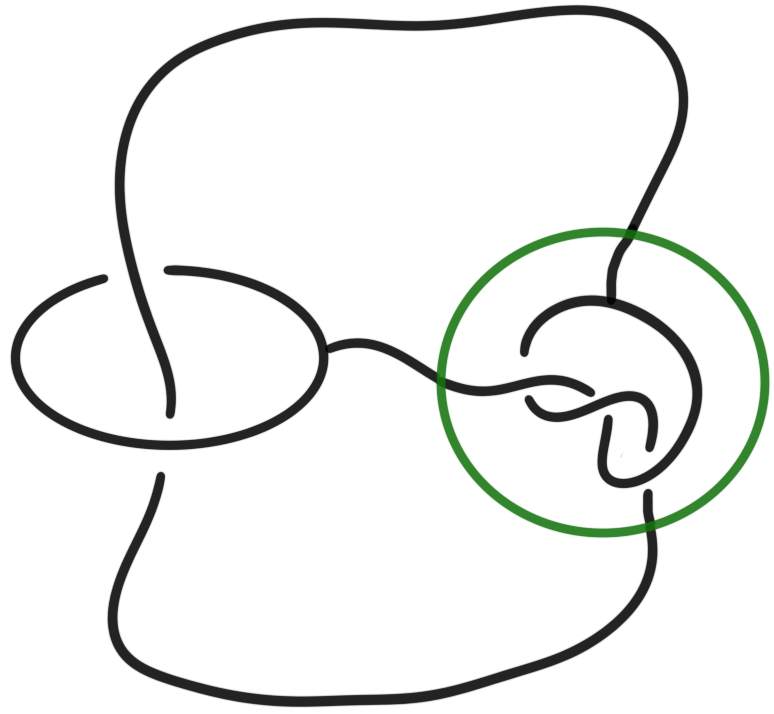}
			\put(84,51.3){\footnotesize $B$} 
		\end{overpic}
		\caption{Irrational $\alpha$.}
		\label{fig:trivial_I}
	\end{subfigure}
	\begin{subfigure}{.24\linewidth}
		\centering
		\begin{overpic}[scale=.09,percent]{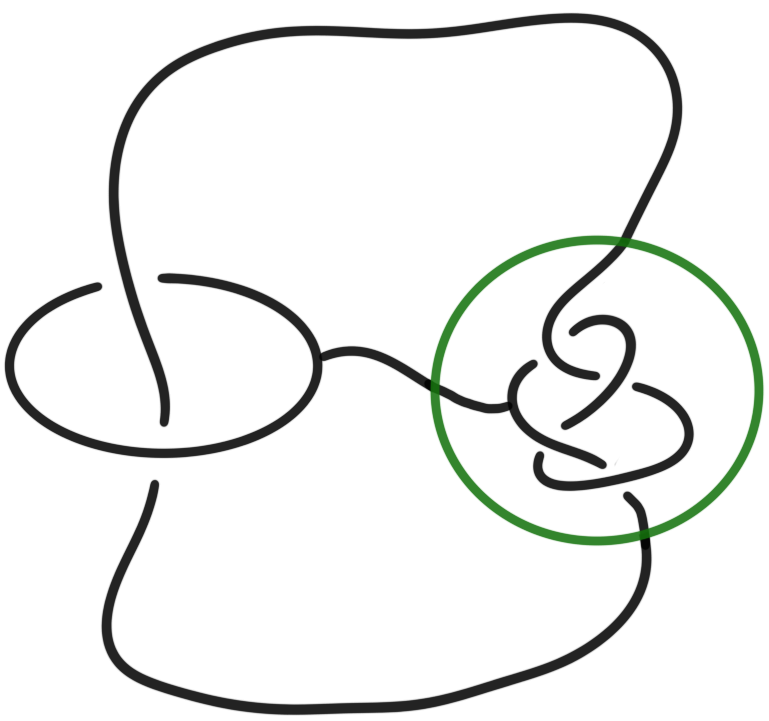}
			\put(86,48){\footnotesize $B$} 
		\end{overpic}
		\caption{Rational $\beta$.}
		\label{fig:trivial_II}
	\end{subfigure} 
	\begin{subfigure}{.24\linewidth}
		\centering
		\begin{overpic}[scale=.09,percent]{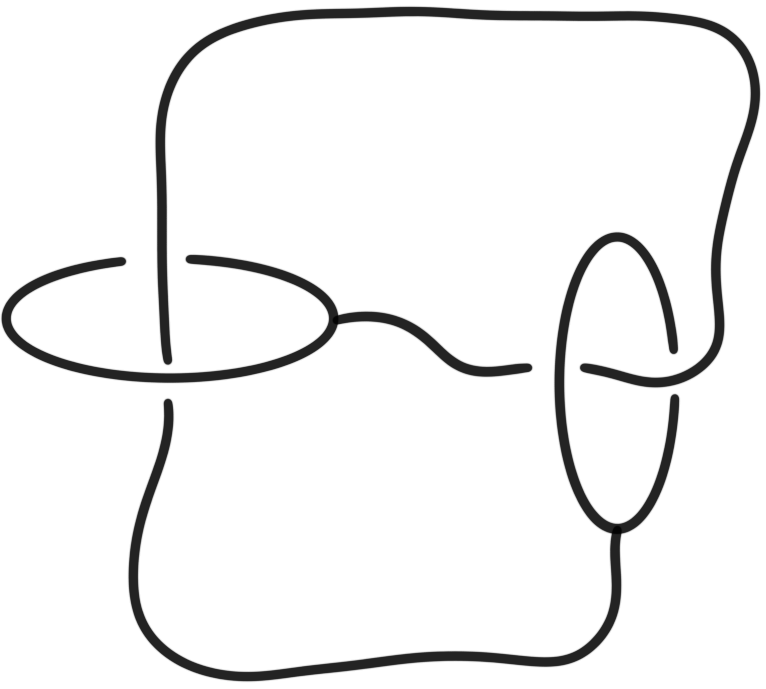}
		\end{overpic}
		\caption{Sum of two $\hctwoone$.}
		\label{fig:hc_fourtwo}
	\end{subfigure}
	\caption{\ref{fig:trivial_I},\ref{fig:trivial_II} are obtained by tangle replacement on \ref{fig:hc_two}.}
	\label{fig:replacement_hc_two} 
\end{figure}

In general, given a spatial graph $\Gamma$, we consider a $3$-ball $B\subset \sphere$ such that $\partial B$ meets $\Gamma$ transversally and $T:=B\cap\Gamma$ is contained in a proper disk in $B$, namely $T$ is \emph{trivial}. A \emph{$T$-tangle} is then a properly embedded CW-complex $\alpha$ in $B$ with $\alpha$ homeomorphic to $T$ and $\alpha\cap \partial B=T\cap \partial B$. The union
$\Gamma_\alpha:=(\Gamma-T)\cup \alpha$ is called \emph{the spatial graph obtained from $\Gamma$ by the tangle replacement with $\alpha\subset B$}; see Fig.\ \ref{fig:small_crossing} for the case where $\Gamma$ is $\hcfourone$ in Fig.\ \ref{fig:hc_fourone}.
We denote by $V_\alpha$ the handlebody-knot given by a regular neighborhood of $\Gamma_\alpha$, and say two handlebody-knots $V_1,V_2$ differ by \emph{a local change to their spines} if $V_1,V_2$ are equivalent 
to $V_\alpha,V_\beta$ for some $T$-tangles $\alpha,\beta$. The handlebody-knots represented by Figs.\ \ref{fig:trivial_I}, \ref{fig:trivial_II} 
and those represented by \ref{fig:fivetwo}, \ref{fig:sixthirteen}, \ref{fig:seventhirtysix}, \ref{fig:seventhirtyeight} are examples. Two $T$-tangles $\alpha,\beta$ in $B$ are \emph{equivalent} if there is an orientation-preserving self-homeomorphism $f$ of $B$ fixing $\partial B$ with $f(\alpha)=\beta$.
With the terminologies, we can now make Question \ref{ques:informal} more specific.

\begin{question}\label{ques:formal}
Given two $T$-tangles $\alpha,\beta$, when are $V_\alpha,V_\beta$ inequivalent?  
\end{question}

For the case $T$ is an arc, the question is answered by Schubert's knot factorization \cite{Sch:49} and Ishii-Kishimoto-Ozawa \cite{IshKisOza:15} and Koda-Ozawa \cite[Appendix $B$]{KodOzaGor:15}. Little is known in general, and the answer depends on the given spatial graph $\Gamma$. 

If $\Gamma$ is the simplest non-trivial spatial handcuff graph $\hctwoone$, namely Fig.\ \ref{fig:hc_two} or $2_1$ in Moriuchi \cite[Table $1$]{Mor:07b}, then the examples given in Figs.\ \ref{fig:trivial_I}, \ref{fig:trivial_II} can be generalized to produce an infinite family of mutually inequivalent $T$-tangles whose corresponding handlebody-knots are all equivalent.
   
If $\Gamma$ is the simplest non-prime, non-trivial spatial handcuff graph, namely, Fig.\ \ref{fig:hc_fourtwo} (see Bellettini-Paolini-Paolini-Wang \cite[Table $4$]{BelPaoPaoWan:23}), the uniqueness of $\mathcal{P}_3$-decomposition in Bellettini-Paolini-Wang \cite{BelPaoWan:25} implies that $\alpha,\beta$ are equivalent if and only if $V_\alpha,V_\beta$ are equivalent. This suggests that the more complicated, for instance, of greater crossing numbers, $\Gamma$ is, 
the more likely inequivalent $T$-tangles yield inequivalent handlebody-knots. However, $\Gamma$ with fewer crossings is of more interest as they can produce handlebody-knots with small crossing numbers. Out of $90$ irreducible handlebody-knots up to seven crossings, more than $20$ (resp. around $5$) can be obtained by tangle replacement on $\hctwoone$ (resp.\ on Fig.\ \ref{fig:hc_fourtwo}).

The present work considers the case
where $\Gamma$ is the \emph{second simplest prime spatial handcuff graph, namely $\hcfourone$} (see Fig.\ \ref{fig:hc_fourone} and Moriuchi \cite[Table $1$]{Mor:07b}), and $T=B\cap \Gamma$ is a cone $\tau$ of three points $x,y,z\subset \partial B$; see Fig.\ \ref{fig:hc_fourone}. A \emph{$\tau$-tangle} is then an embedded CW-complex $\alpha$ in $B$ with $\alpha$ homeomorphic to $\tau$ and $\partial \alpha:=\alpha\cap \partial B=x\cup y\cup z$. The notion of $\tau$-tangle is introduced in Ozawa-Wang \cite{OzaWan:26} to classify $3$-decomposition of genus two handlebody-knots, and tangle replacement considered here can be regarded as a special case of $\tau\rho$-decomposition in \cite{OzaWan:26}. It is also related to the order-$3$ vertex connected sum in Wolcott \cite{Wol:87} (see Moriuchi \cite{Mor:07b}) and the prime decomposition in Motohashi \cite{Moto:07}.

Since $\tau$ is a cone of $x,y,z$, the closure of $\Gamma-\tau$ consists of two components: one is an arc and the other contains a loop; it may be assumed the latter contains $x$. Observe that there is a self-homeomorphism $\pi$ of $(\sphere,\Gamma,B)$, induced by the rotation about the $x$-axis in Fig.\ \ref{fig:hc_fourone}, that fixes $x$ and swaps $y,z$. This motivates the following: An \emph{$x$-equivalence} of two $\tau$-tangles $\alpha,\beta$ is an orientation-preserving self-homeomorphism of $B$ that fixes $x$ and preserves $y\cup z$. An $x$-equivalence is \emph{swapping} if it swaps $y$ and $z$.
Two $\tau$-tangles are \emph{$x$-equivalent} if there is an $x$-equivalence between them. 
Since the pure mapping class group of a $3$-punctured sphere is trivial, the restriction of an $x$-equivalence on $\partial B$ is isotopic either to $\pi$ or to the identity $\id$. In particular, if $\alpha,\beta$ are $x$-equivalent, then $V_\alpha, V_\beta$ are equivalent. The main result, implying the converse, classifies all genus two handlebody-knots obtained by performing tangle replacement on $\hcfourone$. To state the result, we recall the notion of \emph{rational $\tau$-tangles} from Ozawa-Wang \cite{OzaWan:26}. 

Given $s\in\{x,y,z\}$, a \emph{canonical decomposition} of a $\tau$-tangle $\alpha$ at $s$ 
comprises two $3$-balls 
$B_L,B_R$ such that the union $B_L\cup B_R$ is $B$ and the intersection $B_L\cap B_R$ is a disk $D$ with $s=B_L\cap \partial\alpha$ and $\alpha_L:=B_L\cap\alpha$ a trivial $\tau$-tangle in $B_L$; see Fig.\ \ref{fig:rational_example}. 
Set $\alpha_R:=B_R\cap \alpha$. Then $\alpha_R$ is a $2$-string tangle in $B_R$, 
and $\alpha$ is said to be \emph{rational at $s$} if the $2$-string tangle $\alpha_R$ is rational. Note that, if $\alpha$ is rational at $s$, then it is not rational at any other point in $\partial \alpha$. We drop ``at $s$" when $s=x$.

Suppose $\alpha$ is rational at $s\in\partial\alpha$, and $(B_L,\alpha_L)\cup_D (B_R,\alpha_R)$ is the corresponding canonical decomposition at $s$.
Then the boundary of the disk $D=B_L\cap B_R$ gives a longitude of $(B_R,\alpha_R)$, while the meridian is determined only up to some Dehn twist along $\partial D\subset \partial B_R$. Thus $(B_R,\alpha_R)$ can be expressed as in Figs.\ \ref{fig:odd} and \ref{fig:even}. \emph{The slope} $r$ of $\alpha$ is the continued fraction $[a_1,\dots, a_n,0]$, and is well-defined modulo $\mathbb{Z}$. A rational $\tau$-tangle (at $x$) with a slope $r$ is called $r$-rational. For instance, the $\tau$-tangle in Fig.\ \ref{fig:rational_example} is $\frac{1}{3}$-rational. Denote by $V_r$ the handlebody-knot obtained by performing tangle replacement on $\Gamma$ with an $r$-rational tangle.

\begin{figure}[t]
	\begin{subfigure}{.32\linewidth}
		\centering
		\begin{overpic}[scale=.11,percent]{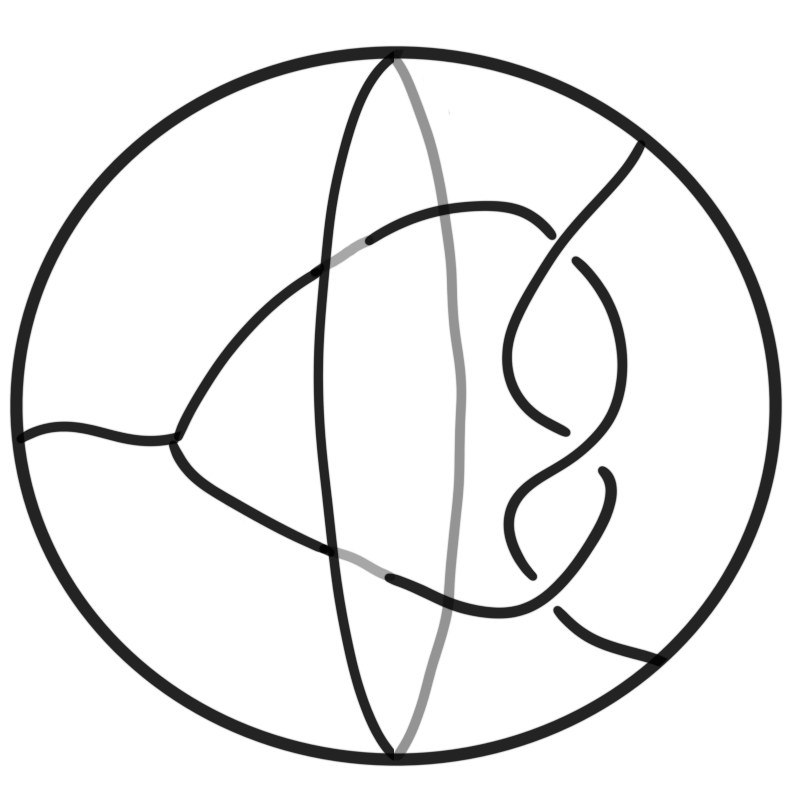}
			\put(46,50){$D$}
			\put(58,80){$B_R$}
			\put(27,80){$B_L$}
		\end{overpic}
		\caption{Rational $\tau$-tangle.}
		\label{fig:rational_example}
	\end{subfigure}
	\begin{subfigure}{.32\linewidth}
		\centering
		\begin{overpic}[scale=.12,percent]{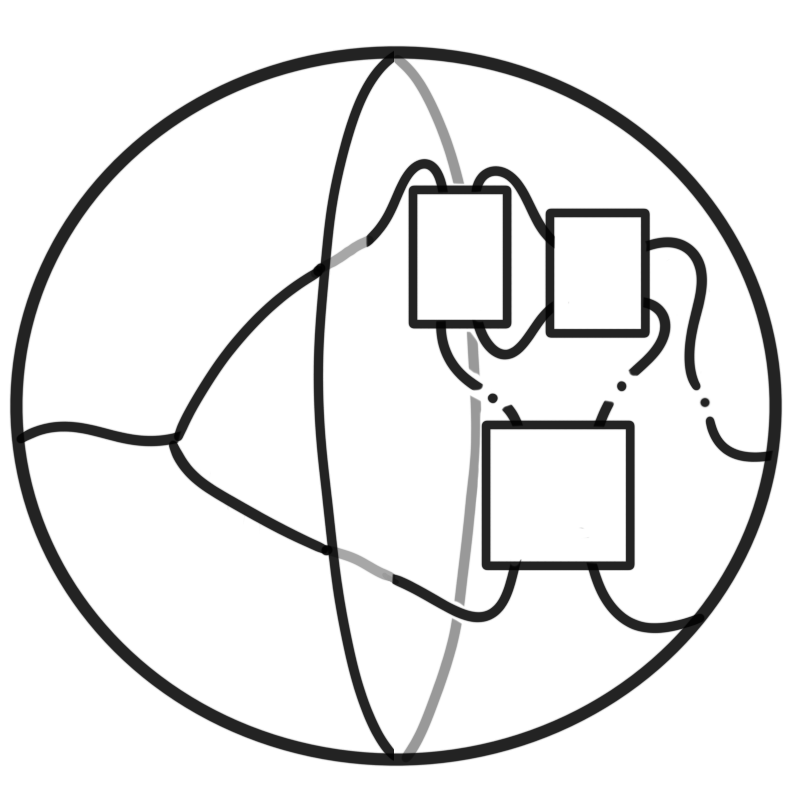}
			\put(67,36){$a_n$}
			\put(71,65){$a_2$}
			\put(54,67){$a_1$}
		\end{overpic}
		\caption{Odd $n$.}
		\label{fig:odd}
	\end{subfigure}
	\begin{subfigure}{.32\linewidth}
		\centering
		\begin{overpic}[scale=.12,percent]{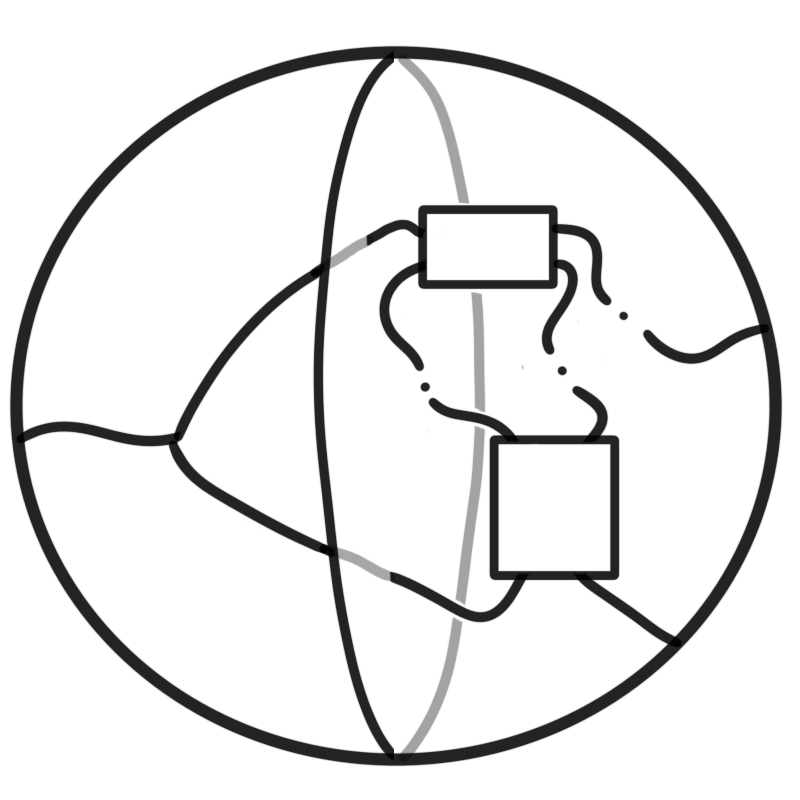}
			\put(58,67){$a_1$}
			\put(65,34){$a_n$}
		\end{overpic}
		\caption{Even $n$.}
		\label{fig:even}
	\end{subfigure} 
	\caption{}
\end{figure} 
  
A $\tau$-tangle $\alpha$ is \emph{atoroidal} if its exterior $\Compl{\alpha}:=B-\openrnbhd{\alpha}$
contains no incompressible torus, where $\openrnbhd{\alpha}$ is the interior of a regular neighborhood of $\alpha$ in $B$.
Let $\mathcal{T}$ be the set of atoroidal $\tau$-tangles up to $x$-equivalence, and $\overline{V}$ be the mirror image of a handlebody-knot $V$. Our first result classifies $V_\alpha$ and determines its chirality.  
 
\begin{theorem}\label{teo:classification}
Members in the family of handlebody-knots:
\[ 
\{
\overline{V_\alpha}, V_\alpha, V_{\frac{1}{5}}=\overline{V_\frac{1}{5}}\mid \alpha\in\mathcal{T}-\{\alpha_{\frac{1}{5}}\}
\} 
\]
are mutually inequivalent. 
\end{theorem} 

A handlebody-knot $V$ is \emph{irreducible} if there is no $2$-sphere in $\sphere$ meeting $V$ in an essential disk. By Tsukui \cite{Tsu:75}, a genus two handlebody-knot is irreducible if and only if its exterior is $\partial$-irreducible. Bellettini-Paolini-Paolini-Wang \cite{BelPaoPaoWan:25p} enumerates and distinguishes all genus two irreducible handlebody-knots of seven crossings except for the two pairs $(\sixtwelve,\seventhirtynine)$ and $(\sevenfiftynine,\sevensixty)$; see Figs.\ \ref{fig:sixtwelve}, \ref{fig:seventhirtynine} and Figs.\ \ref{fig:sevenfiftynine}, \ref{fig:sevensixty}, respectively. 
Each pair has the form $(V_\alpha,\overline{V_{\bar{\alpha}}})$, where $\bar{\alpha}$ is the mirror image of $\alpha$ 
about the plane where the diagram lies. 
For $(\sixtwelve,\seventhirtynine)$, the $\tau$-tangle $\alpha$ is $\frac{2}{5}$-rational; see Fig.\ \ref{fig:two_five_rational} and Figs.\ \ref{fig:V_tau_sixtwelve}, \ref{fig:V_tau_seventhirtynine}, while for $(\sevenfiftynine,\sevensixty)$, $\alpha$ is in Fig.\ \ref{fig:irrational_one_three}, which is not rational at $x$; see Figs.\ \ref{fig:sevenfiftynine}, \ref{fig:sevensixty}. On the other hand, the slope of $\bar{\alpha}$ for an $r$-rational $\tau$-tangle $\alpha$ at $s\in\partial\alpha$ is $-r$, so $V_\alpha$, $\overline{V_{\bar{\alpha}}}$ are never equivalent by Theorem \ref{teo:classification}. Again by Theorem \ref{teo:classification}, if $V_\alpha, V_{\bar{\alpha}}$ are equivalent, then $\alpha, \bar{\alpha}$ are $x$-equivalent, which is not the case when $\alpha$ is $\frac{2}{5}$-rational or the one in Fig.\ \ref{fig:irrational_one_three}. Therefore, we have the corollary.
\begin{corollary}
$\sixtwelve,\seventhirtynine$ (resp.\   $\sevenfiftynine,\sevensixty$) are inequivalent, up to mirror image.   
\end{corollary} 
Note that the pair of handlebody-knots $(\fivetwo, \sixthirteen)$ in Ishii-Kishimoto-Moriuchi-Suzuki \cite{IshKisMorSuz:12} is also of the form $(V_\alpha,\overline{V_{\overline{\alpha}}})$ with $\alpha$ being $\frac{1}{3}$-rational. Theorem \ref{teo:classification} therefore provides an alternative proof for the statement $\fivetwo, \sixthirteen$ are inquivalent, up to mirror image. The two cannot be distinguished by invariants computed in Ishii-Kishimoto-Moriuchi-Suzuki \cite{IshKisMorSuz:12}, and is proved with a geometric method in Lee-Lee \cite{LeeLee:12}. 
Section \ref{subsec:extension} generalizes Theorem \ref{teo:classification} to handlebody-knots obtained by performing tangle replacement on an infinite family spatial graphs; these include all handlebody-knots of the second infinite family in Lee-Lee \cite{LeeLee:12}.

\begin{figure}[t]
	\begin{subfigure}{.24\linewidth}
		\centering
		\begin{overpic}[scale=.1,percent]{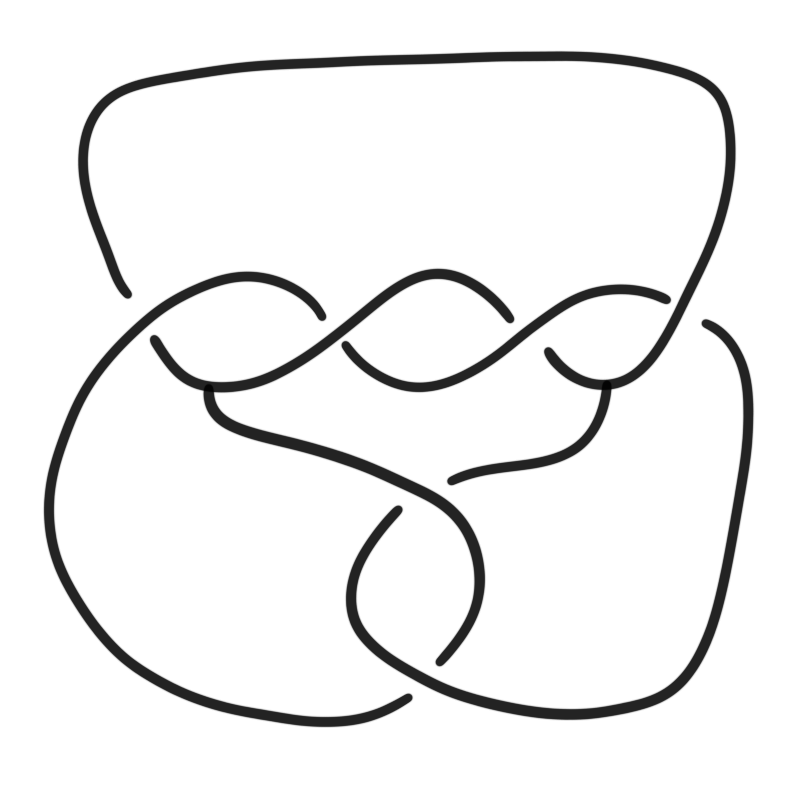}
		\end{overpic}
		\caption{$\sixtwelve$}
		\label{fig:sixtwelve}
	\end{subfigure}
	\begin{subfigure}{.24\linewidth}
		\centering
		\begin{overpic}[scale=.1,percent]{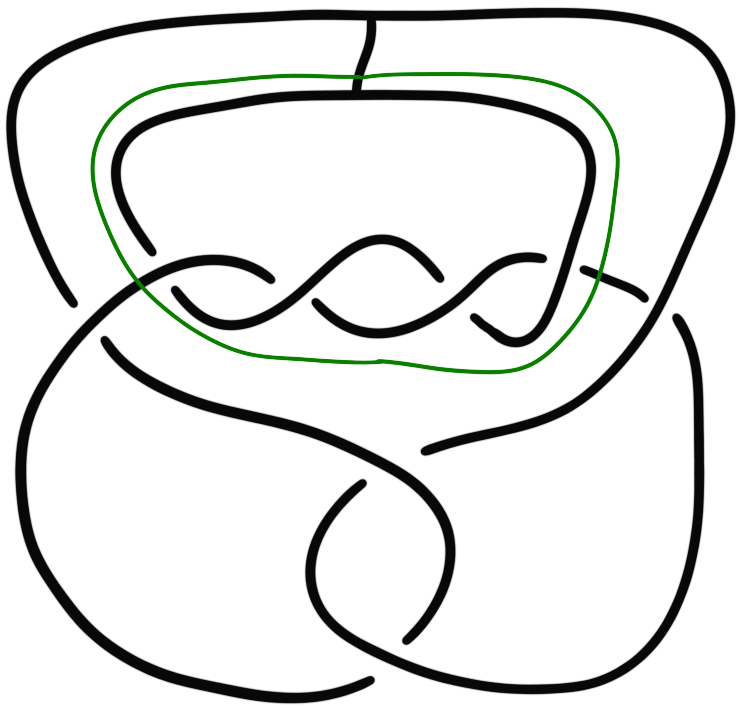}
			\put(6,79){\small $B$}
			\put(51,87){\footnotesize $x$}
			\put(17,49){\footnotesize $y$}
			\put(80,50){\footnotesize $z$}
		\end{overpic}
		\caption{$\sixtwelve$ as $V_\alpha$.}
		\label{fig:V_tau_sixtwelve}
	\end{subfigure}
	\begin{subfigure}{.24\linewidth}
		\centering
		\begin{overpic}[scale=.1,percent]{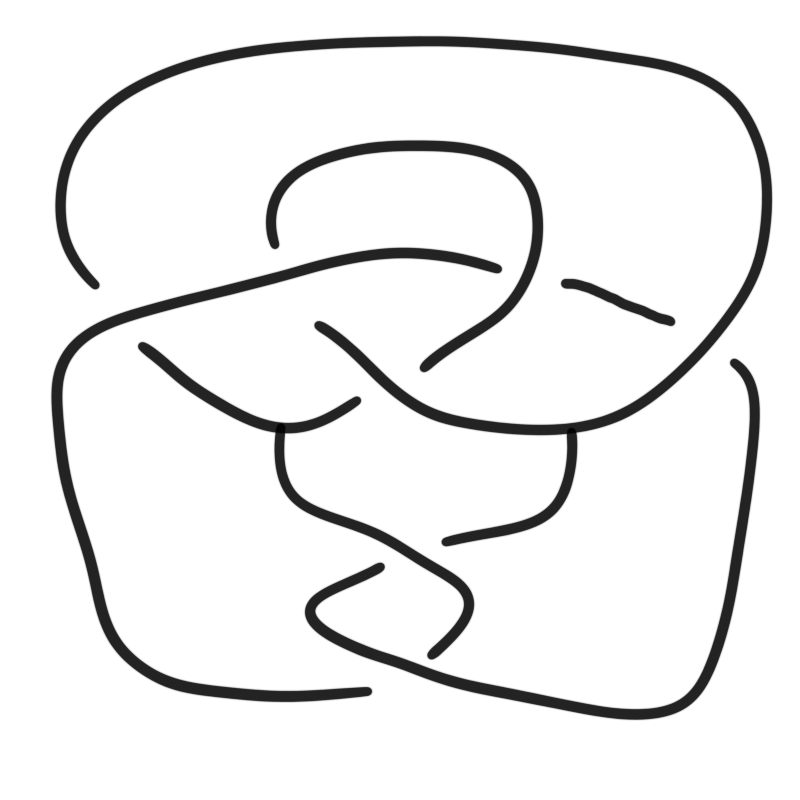}
		\end{overpic}
		\caption{$\seventhirtynine$.}
		\label{fig:seventhirtynine}
	\end{subfigure}  	
	\begin{subfigure}{.24\linewidth}
		\centering
		\begin{overpic}[scale=.1,percent]{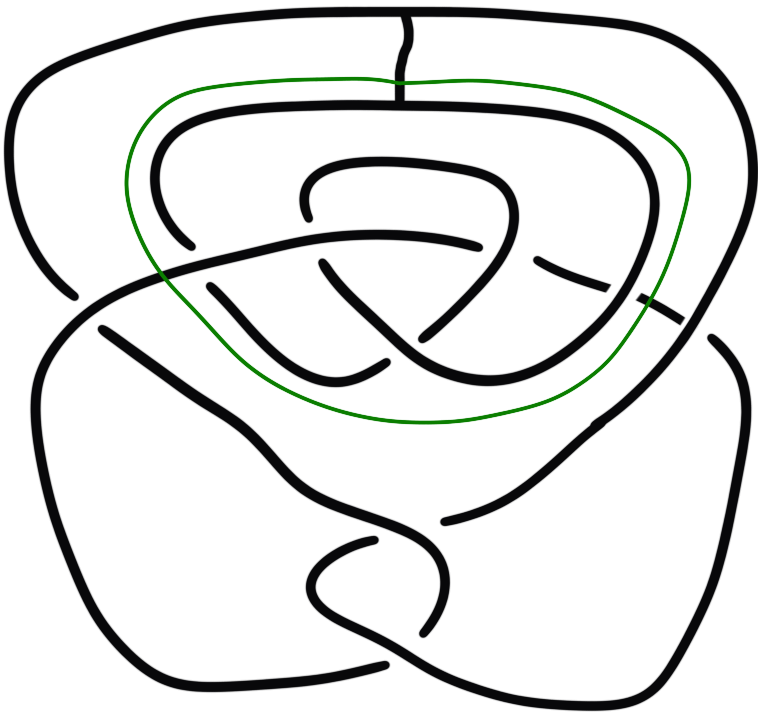}		
		\put(8,75){\small $B$}
		\put(55,84){\footnotesize $x$}
		\put(17,51){\footnotesize $y$}
		\put(88,55){\tiny $z$}
		\end{overpic}
		\caption{$\seventhirtynine$ as $\overline{V_{\bar{\alpha}}}$.}
		\label{fig:V_tau_seventhirtynine}
	\end{subfigure}  	
	\begin{subfigure}{.24\linewidth}
		\centering
		\begin{overpic}[scale=.1,percent]{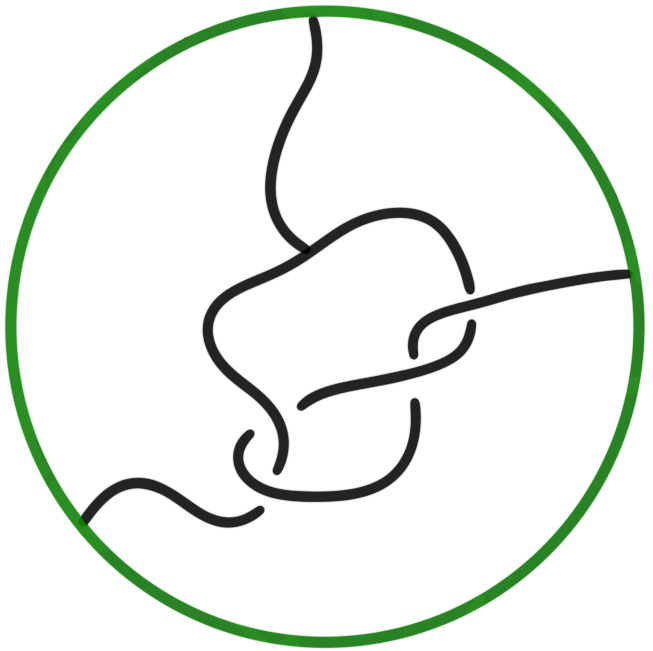}
			\put(10,27){$y$}		
			\put(88,48){$z$}		
			\put(39,89){$x$}		
		\end{overpic}
		\caption{$\frac{2}{5}$-rational $\alpha$.}
		\label{fig:two_five_rational}
	\end{subfigure}
	\begin{subfigure}{.24\linewidth}
		\centering
		\begin{overpic}[scale=.1,percent]{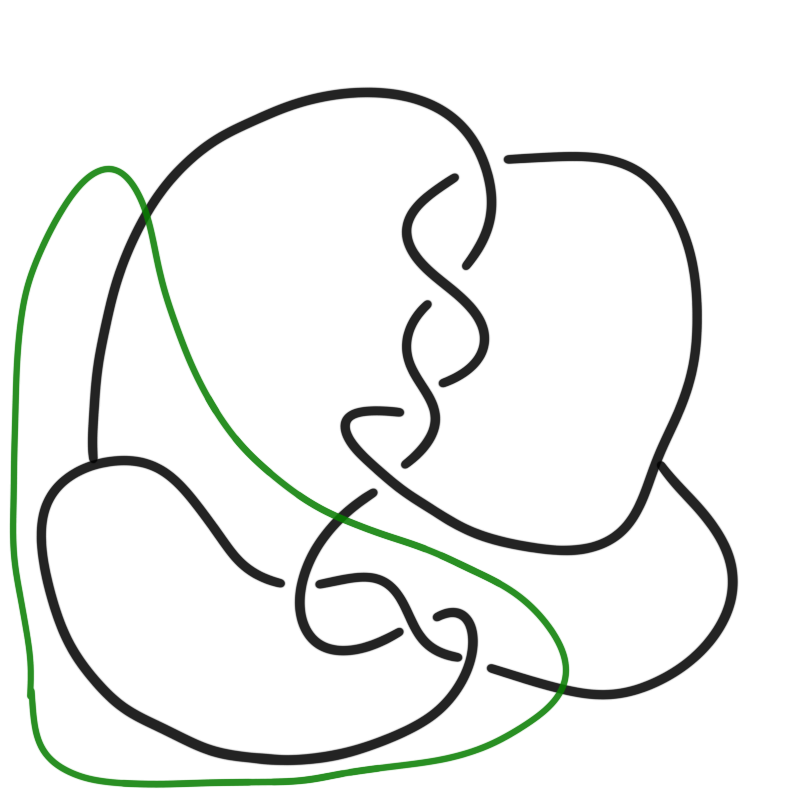}
			\put(14,47){$B$} 	
			\put(71,7){\footnotesize $x$} 	
			\put(40,39){\footnotesize $y$} 	
			\put(20,71){\footnotesize $z$} 	
		\end{overpic}
		\caption{$\sevenfiftynine$ as $V_\alpha$.}
		\label{fig:sevenfiftynine}
	\end{subfigure}
	\begin{subfigure}{.24\linewidth}
		\centering
		\begin{overpic}[scale=.1,percent ]{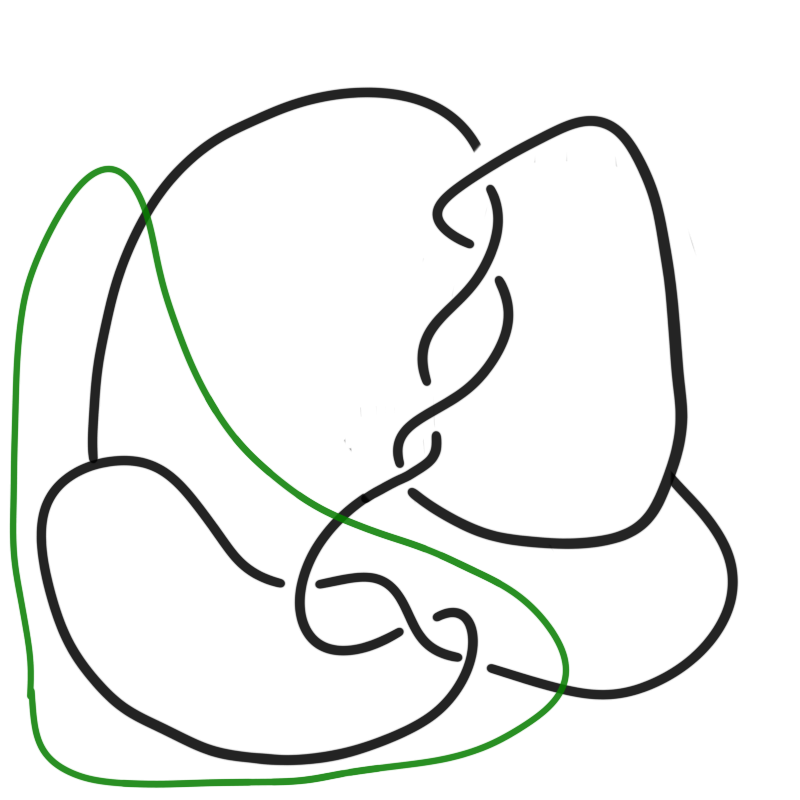}
			\put(14,47){$B$} 	
			\put(71,7){\footnotesize $x$} 	
			\put(40,39){\footnotesize $y$} 	
			\put(20,71){\footnotesize $z$} 	
		\end{overpic}
		\caption{$\sevensixty$ as $\overline{V_{\bar\alpha}}$.}
		\label{fig:sevensixty}
	\end{subfigure}
	\begin{subfigure}{.24\linewidth}
		\centering
		\begin{overpic}[scale=.1,percent]{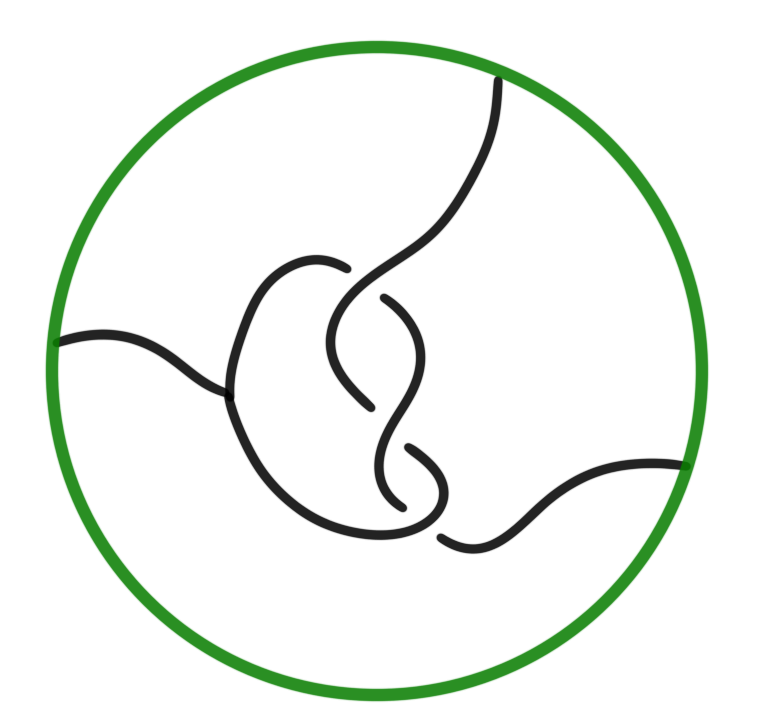}
			\put(83,35){$x$}		
			\put(57,78){$y$}		
			\put(9,52){$z$}		
		\end{overpic}
		\caption{$\alpha$ rational at $z$.}
		\label{fig:irrational_one_three}
	\end{subfigure}
	\caption{}
\end{figure}

Handlebody-knot chirality is in general hard to detect, and as yet, the chirality of genus two handlebody-knots up to six crossings is not completely determined; see Ishii-Iwakiri-Jang-Oshiro. Theorem \ref{teo:classification} implies $V_\alpha$ is chiral if and only if $\alpha$ is not $\frac{1}{5}$-rational. The next result extends this chirality characterization to a classification of the symmetry of $V_\alpha$.
The symmetry of a handlebody-knot $V$ 
is measured by the mapping class group 
\[
\mathrm{MCG}_{(+)}(\sphere,V):=\pi_0(\mathrm{Homeo}_{(+)}(\sphere,V)),
\]
called \emph{the (positive) symmetry group} of $V$, where $\mathrm{Homeo}_{(+)}(\sphere,V)$ stands for the space of (orientation-preserving) self-homeomorphisms of $(\sphere,V)$; see Koda \cite{Kod:15}. Note that $\sym{V}=\psym{V}$ if and only if $V$ is chiral. The symmetry group of classical knots, namely genus one handlebody-knots, has been extensively studied in the past decades; see Kawauchi \cite{Kaw:96}. It is now known that, if the symmetry group of a knot is finite, then it is cyclic or dihedral, and the knot is hyperbolic, torus, or a cable of a torus knot. 
Much less is known about the symmetry group of higher genus handlebody-knots. It is not even clear whether their symmetry groups are always finitely generated, for instance, the Powell conjecture \cite{Pow:80}. However, it follows from Funayoshi-Koda \cite{FunKod:20} that 
the symmetry group of a genus two handlebody-knot $V$ is finite if and only if $V$ is non-trivial and atoroidal, namely, no incompressible torus in the exterior $\Compl V:=\sphere-\mathring{V}$. Our next result classifies the symmetry group of $V_\alpha$. A $\tau$-tangle $\alpha$ is \emph{$x$-symmetric} if there is a swapping $x$-equivalence between $\alpha$ and itself.

\begin{theorem}\label{teo:symmetry} Given an atoroidal $\tau$-tangle $\alpha$, 
\begin{enumerate}[label=(\roman*)]
\item\label{itm:three} if $\alpha$ is $\frac{1}{3}$-rational, then 
\[\mathrm{MCG}(\sphere, V_\alpha)=\mathrm{MCG}_+(\sphere,V_\alpha)=\mathbb{Z}_2\times \mathbb{Z}_2;\]
\item\label{itm:five} if $\alpha$ is $\frac{1}{5}$-rational, then 
\[\mathrm{MCG}(\sphere, V_\alpha)=\mathbb{Z}_2\times \mathbb{Z}_2,\quad \mathrm{MCG}_+(\sphere, V_\alpha)=\mathbb{Z}_2;\]
\item\label{itm:x_symmetric} if $\alpha$ is $x$-symmetric but not $\frac{1}{3}$- or $\frac{1}{5}$-rational, then 
\[\mathrm{MCG}(\sphere, V_\alpha)= \mathrm{MCG}_+(\sphere, V_\alpha)=\mathbb{Z}_2;\]
\item\label{itm:not_x_symmetric} if $\alpha$ is not $x$-symmetric, then 
\[\mathrm{MCG}(\sphere, V_\alpha)= \mathrm{MCG}_+(\sphere, V_\alpha)=\Zone,\]
\end{enumerate}
where $\Zone$ is the trivial group. 
\end{theorem}

As an application, we determine the symmetry group of $\sixtwelve$. This, together with the results from Koda \cite{Kod:15}, Wang \cite{Wan:22, Wan:23, Wan:24ii}, Bellettini-Paolini-Wang \cite{BelPaoWan:25}, Koda-Ozawa-Wang \cite{KodOzaWan:25i}, and  Ozawa-Wang \cite{OzaWan:26}, determines the symmetry groups of all atoroidal, non-hyperbolic handlebody-knots, up to six crossings; see Table \ref{tab:symmetry_up_to_six}. 
The symmetry groups of several seven-crossing handlebody-knots are also computed; see Table \ref{tab:symmetry_seven}. The two tables follow from Theorem \ref{teo:symmetry} because, viewing them as $V_\alpha$, up to mirror image, we have $\alpha$ is $\frac{2}{5}$-rational for $\sixtwelve, \seventhirtynine$ (Fig.\ \ref{fig:two_five_rational}) and is $\frac{2}{7}$- and $\frac{3}{7}$-rational for $\seventhirtysix, \seventhirtyeight$ (Figs.\ \ref{fig:seventhirtysix}, \ref{fig:seventhirtyeight}), and rational $\tau$-tangles are $x$-symmetric.
For $\sevenfiftynine, \sevensixty$, the $\tau$-tangle $\alpha$ is not $x$-symmetric (Fig.\ \ref{fig:irrational_one_three}).

\begin{center}
	\begin{table}[t]
		\caption{Table of symmetry groups.}
		\parbox{.45\linewidth}{
			\centering 
			\subcaption{Up to six crossings.}
			\label{tab:symmetry_up_to_six}
			\begin{tabular}{ |c|c|c|c| }
				\hline
				& $\mathrm{MCG}_+$ & $\mathrm{MCG}$ &  \\ 
				\hline 
				$\mathbf{4_1}$ & $\Ztwo$ & $\Ztwo\times \Ztwo$ & \cite{Wan:24ii} \\ 
				\hline
				$\mathbf{5_1}$ & $\Zone$ & $\Zone$ & \cite{Kod:15}\\
				\hline
				$\mathbf{5_2}$ & $\Ztwo\times \Ztwo$ & $\Ztwo\times \Ztwo$ & \cite{Wan:23}\\
				\hline
				$\mathbf{6_1}$ & $\Zone$ & $\Zone$ & \cite{Kod:15}\\
				\hline
				$\mathbf{6_4}$ & $\Ztwo$ & $\Ztwo$ & \cite{Wan:23}\\
				\hline
				$\mathbf{6_{10}}$ & $\Ztwo$ & $\Ztwo$ & \cite{BelPaoWan:25}\\
				\hline	
				$\mathbf{6_{11}}$ & $\Zone$ & $\Zone$ & \cite{Wan:22}\\
				\hline	
				$\mathbf{6_{12}}$ & $\Ztwo$ & $\Ztwo$ & \\
				\hline	
				$\mathbf{6_{13}}$ & $\Ztwo$ & $\Ztwo$ & 
				\cite{KodOzaWan:25i}\\
				\hline
			\end{tabular}
		}
		\parbox{.45\linewidth}{
			\centering 
			\subcaption{Seven crossings.}
			\label{tab:symmetry_seven}
			\begin{tabular}{ |c|c|c| }
				\hline
				& $\mathrm{MCG}_+$ & $\mathrm{MCG}$ \\
				\hline 
				$\mathbf{7_{32}}$ & $\Ztwo$ & $\Ztwo\times \Ztwo$ \\ 
				\hline 
				$\mathbf{7_{36}}$ & $\Ztwo$ & $\Ztwo$ \\ 
				\hline
				$\mathbf{7_{38}}$ & $\Ztwo$ & $\Ztwo$ \\ 
				\hline		
				$\mathbf{7_{39}}$ & $\Ztwo$ & $\Ztwo$ \\ 
				\hline	
				$\mathbf{7_{59}}$ & $\Zone$ & $\Zone$ \\
				\hline
				$\mathbf{7_{60}}$ & $\Zone$ & $\Zone$ \\
				\hline
			\end{tabular}
		}	
	\end{table}
\end{center}

\subsection*{Conventions}
We work in the piecewise linear category. Given a subpolyhedron $X$ of a $3$-manifold $M$, $\mathring{X}$ and $\mathfrak{N}(X)$ denote the interior and a regular neighborhood of $X$ in $M$, respectively.
The \emph{exterior} $\Compl X$ of $X$ in $M$ is defined to be $M-\openrnbhd{X}$ 
if $X\subset M$ is of positive codimension, and to be $M-\mathring{X}$ otherwise. 
By $\vert X\vert$, we understand the number of components in $X$. Submanifolds are assumed to be proper and in general position.
A surface $S\subset M$ with non-positive Euler characteristic is \emph{essential} if it is incompressible, $\partial$-incompressible and non-boundary-parallel.

\section{Tangle replacement on $\hcfourone$}\label{sec:rigidity}
\begin{figure}[b]
	\begin{subfigure}[b]{.25\linewidth}
		\centering
		\begin{overpic}[scale=.1,percent]{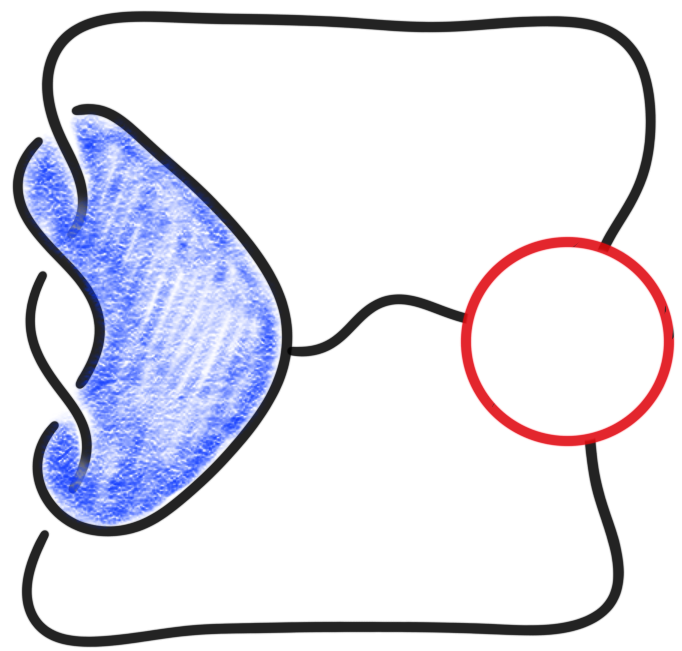}
			\put(25,43){$S$}
			\put(80,43){\Large $\alpha$} 
		\end{overpic}
		\caption{Spine $\Gamma_\alpha$ of $V_\tau$.}
		\label{fig:fourone_tau}
	\end{subfigure} 
	\begin{subfigure}[b]{.73\linewidth}
		\centering
		\begin{overpic}[scale=.1,percent]{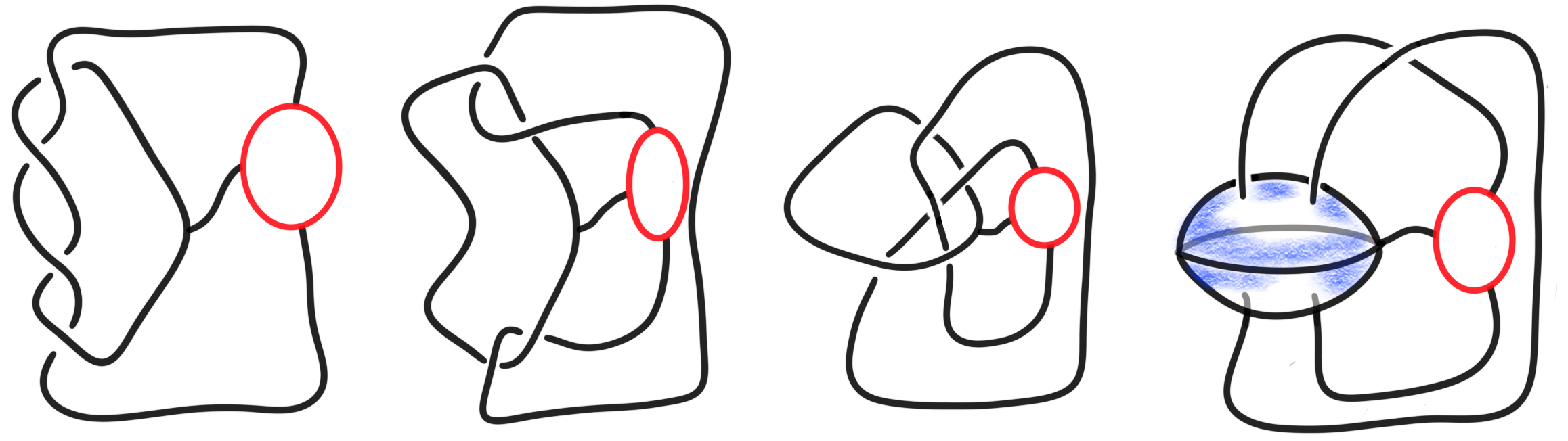}
			\put(17.5,16){\Large $\alpha$}
			\put(40.9,15){$\alpha$}
			\put(65.7,14){$\alpha$}
			\put(93,11.8){$\alpha$}
			\put(79.6,13.7){\footnotesize $S_+$}
			\put(79.7,8.6){\tiny $S_-$}
		\end{overpic}
		\caption{Deform and split.}
		\label{fig:deformation}
	\end{subfigure}
	\begin{subfigure}[b]{.36\linewidth}
		\centering
		\begin{overpic}[scale=.15,percent]{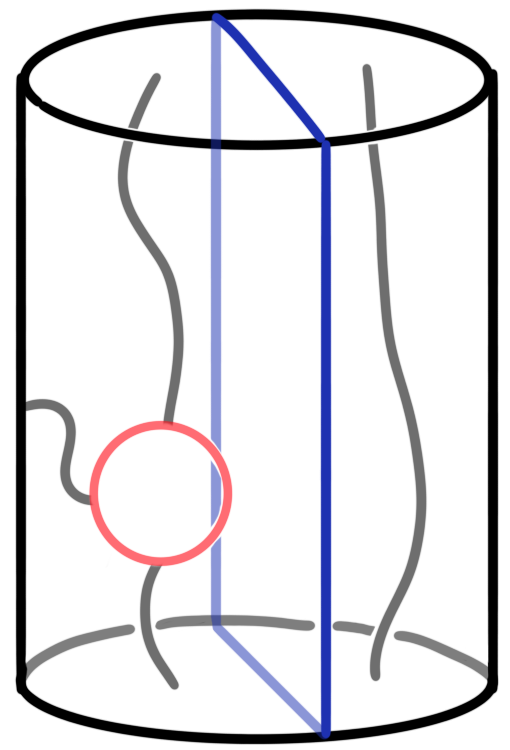}
			\put(19,90.5){\tiny $b_+$}
			\put(45.7,91.5){\tiny $a_+$}
			\put(60,95){\tiny $c_+$}
			\put(23.5,6.5){\tiny $a_-$}
			\put(51,6.7){\tiny $b_-$}
			\put(63,2.5){\tiny $c_-$}
			\put(31,50){$Q$}
			\put(7.2,87){\footnotesize $S_+$}
			\put(7.2,7.5){\footnotesize $S_-$}
			\put(18,31){$\alpha$}
		\end{overpic}
		\caption{Rectangle $Q$ in $\Split{V_\alpha}$.}
		\label{fig:split_manifold}
	\end{subfigure}
	\begin{subfigure}{.36\linewidth}
		\centering
		\begin{overpic}[scale=.15,percent]{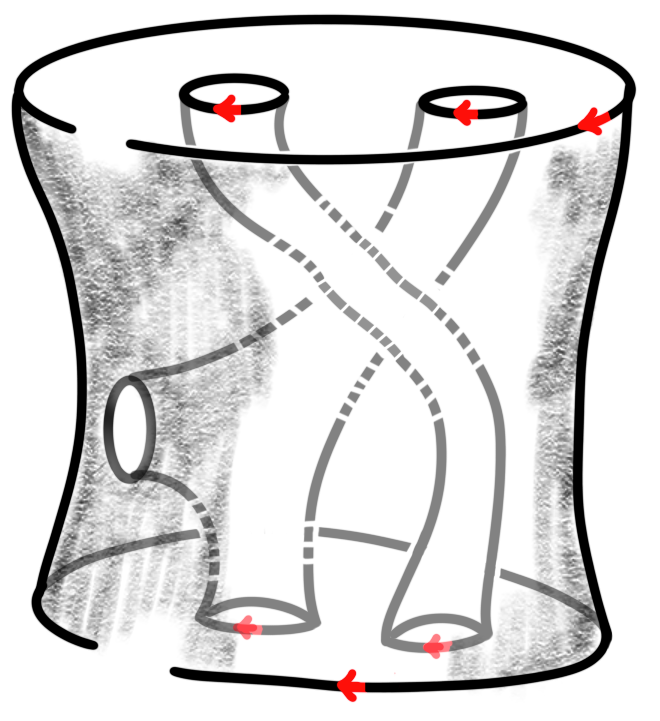}
			\put(40,89){\footnotesize $a_+$}
			\put(70,88.7){\footnotesize $b_+$}
			\put(10.4,81.2){\footnotesize $c_+$}		\put(40,7.4){\footnotesize $a_-$}
			\put(71,9){\footnotesize $b_-$}
			\put(15.2,7){\footnotesize $c_-$}
			\put(25,55){$P$}
			\put(62.5,30){\transparent{.8}{$F$}}		
		\end{overpic}
		\caption{Orientation of $a_\pm, b_\pm, c_\pm$.}
		\label{fig:orientation}
	\end{subfigure}
	\begin{subfigure}{.25\linewidth}
		\centering
		\begin{overpic}[scale=.12,percent]{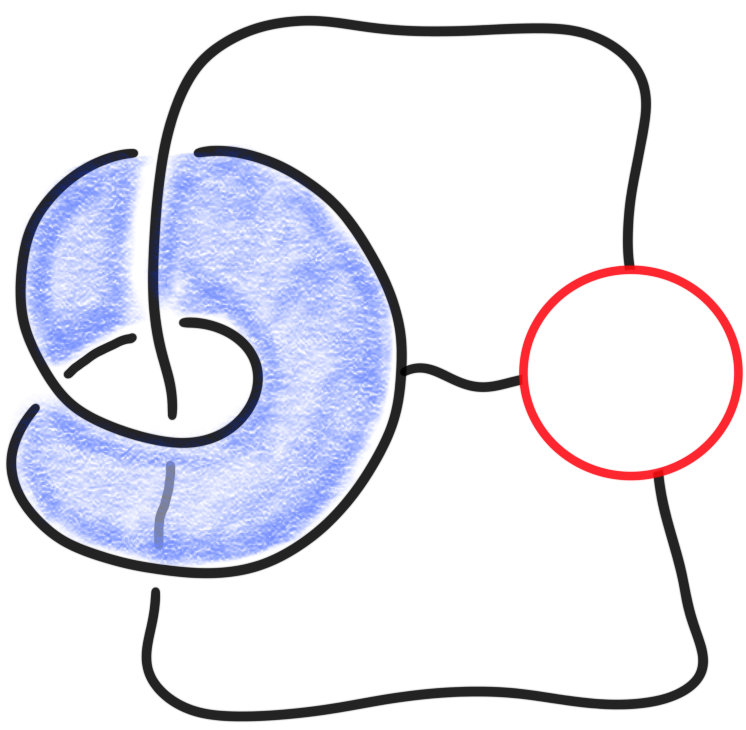}
			\put(25,62){\small $M$}
			\put(81,45){\Large $\alpha$}
		\end{overpic}
		\caption{M\"obius band $M_\ast$.}
		\label{fig:mobius}
	\end{subfigure}
	
	\caption{}
\end{figure}

Throughout the section, $\Gamma_\alpha$ denotes the spatial graph obtained by performing tangle replacement on $\hcfourone$ with a $\tau$-tangle $\alpha\subset B$ (see Fig.\ \ref{fig:fourone_tau}),  and $V_\alpha$ denotes the induced handlebody-knot. Let $S_B$ be the pair of pants $\Compl{V_\alpha}\cap \partial B$ in the exterior $\Compl{V_\alpha}$ of $V_\alpha$ in $\sphere$, and note that $S_B\subset \Compl{V_\alpha}$ is incompressible.

\subsection{The exterior of $V_\alpha$}
Here we study the topology of the exterior $\Compl{V_\alpha}$, and determine the $\partial$-irreducibility and atoroidality of $\Compl{V_\alpha}$, and classify essential annuli it admits in terms of $\alpha$. Recall that $\Compl{\alpha}$ is the exterior of $\alpha$ in $B$.

\subsubsection{Atoroidality and irreducibility}
\begin{lemma}\label{lm:atoroidality}
The $\tau$-tangle $\alpha$ 
is atoroidal if and only if the handlebody-knot $V_\alpha$ is atoroidal.		
\end{lemma}
\begin{proof} 
\textbf{Case $1$: $\alpha$ is trivial.} Then $V_\alpha$ is trivial, so both $\Compl{\alpha}$, $\Compl{V_\alpha}$ are handlebodies, and therefore atoroidal. 

\textbf{Case $2$: $\alpha$ is non-trivial.} If there is an incompressible torus $T$ in $\Compl \alpha$, then, by the incompressibility of $S_B$, every compressing disk of $T$ in $\Compl{V_\alpha}$ can be isotopied into $\Compl\alpha$, a contradiction, so $T$ is incompressible also in $\Compl{V_\alpha}$. 

Conversely, if there is an incompressible torus in $\Compl{V_\alpha}$, 
then we can choose one, say $T$, so that $\vert T\cap \partial B \vert $ is minimized. By the incompressibility of $S_B$, $B\cap T$ consists of some annuli. Since $\alpha$ is a $\tau$-tangle, 
every annulus $A\subset B\cap T$ has parallel boundary components in $S_B$, and hence $A$ is either $\partial$-parallel in $\Compl\alpha$ or cut off from $\Compl\alpha$ a non-trivial knot exterior; either case contradicts the minimality. 
As a result, $T$ is disjoint from $\partial B$. Since $V_\alpha\cup B$ is a trivial handlebody-knot, $T$ is necessarily in $B$, and hence is an incompressible torus in $\Compl \alpha$.   
\end{proof}
 
\begin{lemma}\label{lm:non_trivial_essen_irreducible}
Suppose $\alpha$ is atoroidal. Then 
the following are equivalent: 
\begin{enumerate}[label=(\roman*)]
\item\label{itm:X_non_trivial} $\alpha$ is non-trivial.
\item\label{itm:SB_essen} $S_B$ is essential.
\item\label{itm:irreducible} $V_\alpha$ is an irreducible handlebody-knot.
\end{enumerate}
\end{lemma}
\begin{proof}
\ref{itm:SB_essen}$\Rightarrow$\ref{itm:irreducible}: If $V_\alpha$ is reducible, then by Lemma \ref{lm:atoroidality}, $V_\alpha$ is atoroidal. By Tsukui \cite{Tsu:75}, $V_\alpha$ is trivial and thus $\Compl{V_\alpha}$ is a handlebody. The only essential surfaces in a handlebody are meridian disks, so $S_B$ is inessential.
\ref{itm:irreducible}$\Rightarrow$\ref{itm:X_non_trivial}: If $\alpha$ is trivial, then $V_\alpha$ is trivial and hence reducible. 

\ref{itm:X_non_trivial}$\Rightarrow$ \ref{itm:SB_essen}: If $S_B$ is inessential, then $S_B$ is $\partial$-compressible, and every $\partial$-compressing disk $D$ of $S_B$ is necessarily in $B$. If $\partial D$ meets two components of $S_B$, cutting $B$ along $D$ induces a $3$-ball $B'$ meeting $\alpha$ at an arc, by the atoroidality, $(B',\alpha\cap B')$ is a trivial ball-arc pair, so $\alpha$ is trivial. Similarly, if $\partial D$ meets only one component of $S_B$, then cutting $B$ along $D$ gives us two $3$-balls $B',B''$, each of which meets	$\alpha$ at a trivial arc, again by the atoroidality. Therefore $\alpha$ is trivial.
\end{proof}

\subsubsection{Essential annuli}
Hereinafter $\tau$-tangles are assumed to be \emph{atoroidal and non-trivial}. 
Consider now the pair of pants $S\subset \Compl{V_\alpha}$ in Fig.\ \ref{fig:fourone_tau}, and denote by $a,b,c$ 
the components of $\partial S$ with $c$ the component not bounding any disk in $V_\alpha$. Note also that $S\subset\Compl {V_\alpha}$ is incompressible.
Denote by $\Split {V_\alpha}$ the $3$-manifold obtained by 
splitting $\Compl{V_\alpha}$ along $S$. The boundary of $\Split{V_\alpha}$ contains two copies $S_+,S_-$ of $S$; see Figs.\ \ref{fig:deformation}, 
\ref{fig:split_manifold}.
The boundary $\partial S$ splits the boundary $\partial V_\alpha$ into a sphere $P$ with four disjoint open disks removed and an annulus $F$, each meeting both $S_+,S_-$; see Fig.\ \ref{fig:orientation}.

Denote by $a_\pm,b_\pm,c_\pm$ the components of $\partial S_\pm$ with $x_\pm$ corresponding to the same component $x$ in $\partial S$, where $x=a,b$ or $c$. 
Then it may be assumed that $\partial P=c_+\cup c_-\cup a_-\cup b_+$, $\partial F=a_+\cup b_-$. We orient $a_\pm,b_\pm,c_\pm$ so that
$[a_\pm]+[b_\pm]=[c_\pm]\in H_1(\Split{V_\alpha})$; see Fig.\ \ref{fig:orientation}.

\begin{lemma}\label{lm:S_disjoint_annulus}
The pair of pants $S$ is essential and no essential annuli in $\Compl{V_\alpha}$ are disjoint from $S$. 
\end{lemma}
\begin{proof}
If $S$ is inessential, then $S$ is $\partial$-compressible. Choose a $\partial$-compressing disk $D$ of $S$ 
that minimizes $\vert D\cap S_B\vert$. 
Since $F$ is an annulus meeting both $S_+,S_-$,
the disk $D$ is disjoint from $F$. Without loss of generality, it may be assumed that $D$ meets $S_+$, and hence $D\cap S_-=\emptyset$. Particularly, $D$ can be regarded as a disk in the knot exterior given by 
attaching two $2$-handles to $\Split{V_\alpha}$ along $a_-$ and $b_-$. 
Since $D$ is disjoint from $a_+$, either $D$ meets both $b_+, c_+$ or only one of them.

If it is the former, the disk $D$ meets $S_B$ because every arc connecting $b_+$ and $c_+$ in $P$ meets $B$.  
If it is the latter, then $D$ is separating. Therefore $D\cap P$ is a separating arc in $P$ either meeting $b_+$ and separating $c_+$ from $c_-\cup a_-$ or meeting $c_+$ and separating $b_+$ from $c_-\cup a_-$. Either case implies 
the disk $D$ meets $S_B$. 

By the minimality, for every disk $D'$ in $D\cap B$, the intersection $(P\cap B)\cap \partial D'$ is essential in $P\cap B$, so $D'$ is a $\partial$-compressing disk of $S_B$, contradicting $\Compl \alpha$ is non-trivial by Lemma \ref{lm:non_trivial_essen_irreducible}. This proves the first assertion. 

Suppose there is an annulus $A$ disjoint from $P$. Then observe first that $\partial A$ cannot be in $F$ by the atoroidality of $\Compl{V_\alpha}$, given Lemma \ref{lm:atoroidality}. On the other hand, 
because any three of the homology classes $[b_+],[a_-],[c_+],[c_-]$ form a basis of $H_1(\Split{V_\alpha})$, $\partial A$ cannot both be in $P$. Therefore one component of $\partial A$ is in $F$, and the other in $\partial P$. This forces the loop $P\cap \partial A$ to be the one separating $c_+\cup b_+$ from $c_-\cup a_-$. In particular, $A$ can be isotoped so 
that $A\cap B$ is a disk $D$. Since $(P\cap B)\cap \partial D$ is an essential arc in $P\cap B$, $D$ is a $\partial$-compressing disk of $S_B$, contradicting that $\Compl \alpha$ is non-trivial by Lemma \ref{lm:non_trivial_essen_irreducible}.
\end{proof}

Observe that $P\cup F\subset \partial \Split{V_\alpha}$ distinguishes itself from $S_+\cup S_-\subset \partial\Split{V_\alpha}$ as being the subsurface induced from $\partial V_\alpha$. 
In general, a \emph{$3$-manifold pair} $(X,K)$ 
is an oriented $3$-manifold $X$ with a compact, not necessarily connected, subsurface $K\subset \partial X$.
Two $3$-manifold pairs $(X_1,K_1),(X_2,K_2)$ 
are \emph{equivalent} if there is an orientation-preserving homeomorphism $f:X_1\rightarrow X_2$ with $f(K_1)=K_2$.

For instance, $(\Split{V_\alpha},P\cup F)$ is a $3$-manifold pair. Given by a $\tau$-tangle $\alpha$, there is an associated $3$-manifold pair $(\Compl{\alpha}, F_\tau)$ given by the exterior $\Compl{\alpha}$, together with the frontier $F_\alpha$ of $\rnbhd{\alpha}$ in $B$; see Fig.\ \ref{fig:tau_exterior}.

In view of Lemma \ref{lm:S_disjoint_annulus}, 
$S$ cuts every essential annulus in $\Compl{V_\alpha}$ into some disks, each of which meets $S_+\cup S_-$ at two arcs. This motivates the following notion introduced in \cite{OzaWan:26}. A \emph{rectangle} in a $3$-manifold pair $(X,K)$ is a proper disk $R\subset X$ with $\partial R$ meeting $K$ transversally at two arcs, and 
it is \emph{good} if $R\cap K\subset K$ are two essential arcs. Good rectangles in the $3$-manifold pair associated to a $\tau$-tangle $\alpha$ are classified in \cite{OzaWan:26}. More specifically, 
let $F_\alpha':=\partial \Compl \alpha- \mathring{F}_\alpha$, and denote by $x,y,z$ the components of $\partial F_\alpha=\partial F_\alpha'$. 
Then an $\{st,uv\}$-rectangle is a good rectangle $R$ in $(\Compl{\alpha},F_\alpha)$ with one arc in $\partial R\cap F_\tau'$ connecting the components $s,t$ and the other connecting $u,v$, 
where $s,t,u,v\in\{x,y,z\}$; see Figs.\ \ref{fig:xyxz}, \ref{fig:xxxy}.
\begin{figure}[h]
 \begin{subfigure}{.32\linewidth}
		\centering
		\begin{overpic}[scale=.1,percent]{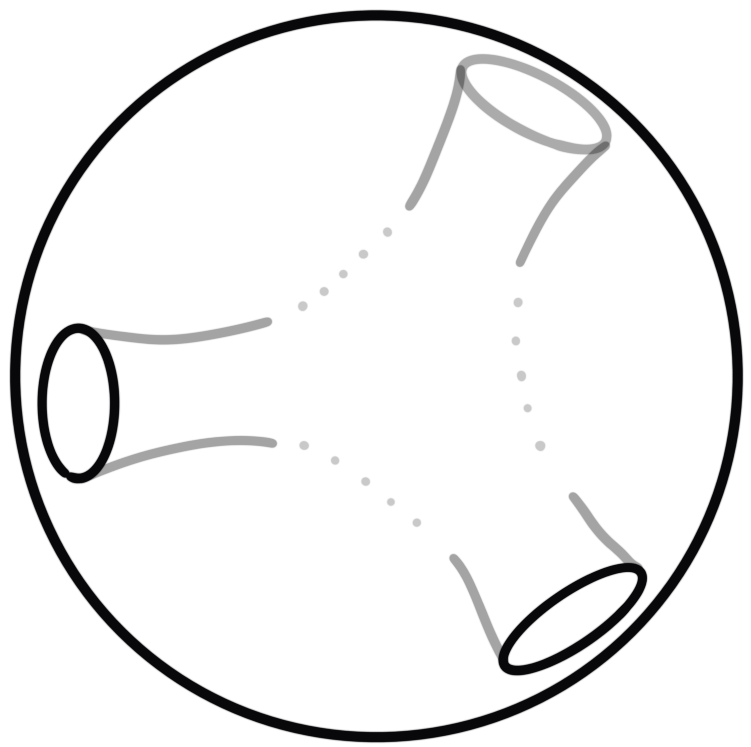}
			\put(20,44){\footnotesize \transparent{.7}{$F_\alpha'$}} 	
			\put(30,70){\footnotesize $F_\alpha$} 	
			\put(8,41){\footnotesize $x$} 	
			\put(70,83){\footnotesize \transparent{.6}{$y$}} 	
			\put(73,14){\footnotesize $z$} 	
		\end{overpic}
		\caption{The pair $(\Compl{\alpha},F_\alpha)$.}
		\label{fig:tau_exterior}
\end{subfigure}
	\begin{subfigure}{.32\linewidth}
		\centering
		\begin{overpic}[scale=.1,percent  ]{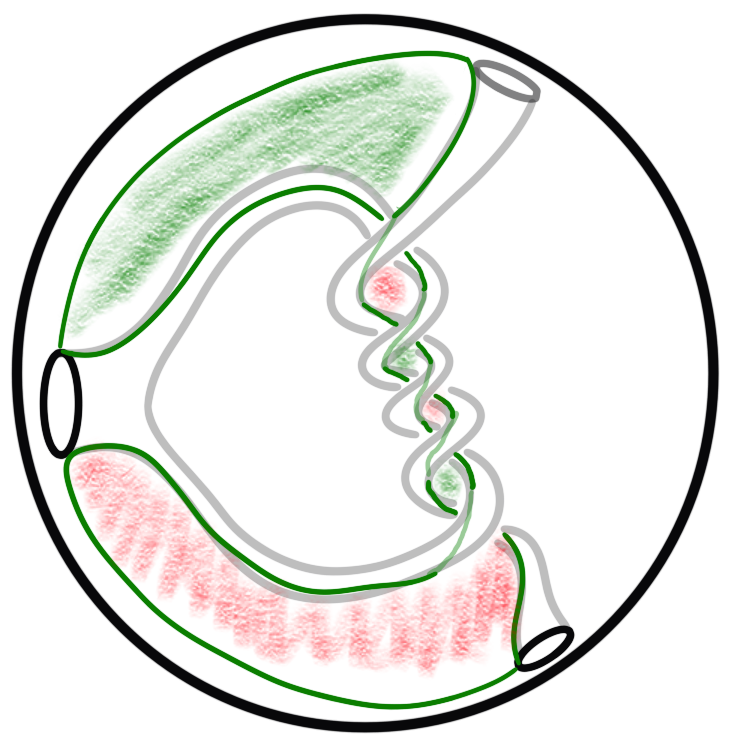}
			\put(-5,43){\footnotesize $x$} 	
			\put(70,94){\footnotesize $y$} 	
			\put(75,3){\footnotesize $z$} 	
		\end{overpic}
		\caption{$\{xy,xz\}$-rectangle.}
		\label{fig:xyxz}
	\end{subfigure}
	\begin{subfigure}{.32\linewidth}
		\centering
		\begin{overpic}[scale=.1,percent]{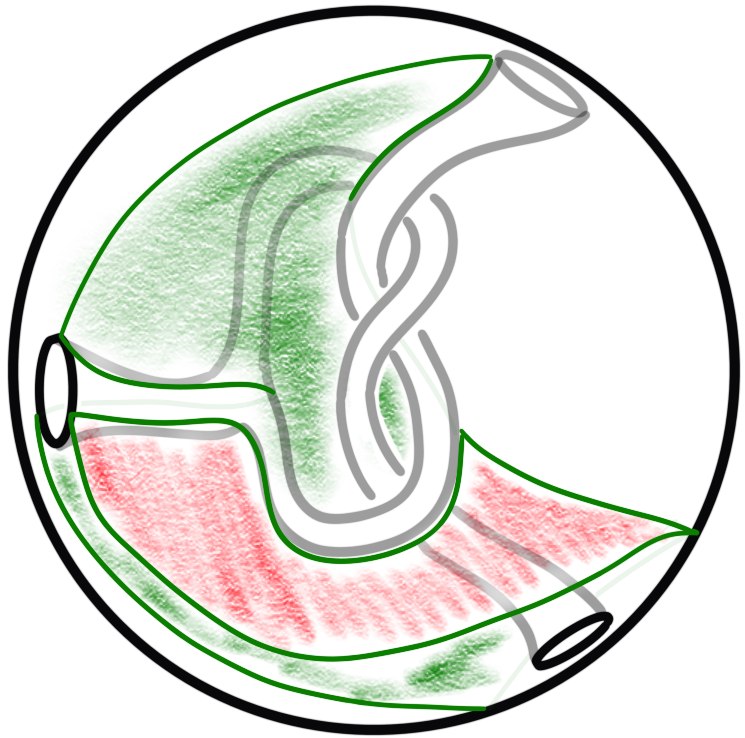}
			\put(-5,43){\footnotesize $x$} 	
			\put(72,94){\footnotesize $y$} 	
			\put(79,3){\footnotesize $z$} 			
			\end{overpic}
		\caption{$\{xx,xy\}$-rectangle.}
		\label{fig:xxxy}
	\end{subfigure}
	\caption{Good rectangles in the $3$-manifold pair $(\Compl{\alpha},F_\alpha)$.}
\end{figure} 
  
\begin{lemma}[\cite{OzaWan:26}]\label{lm:rectangles_tau_tangle} 
Given a $\tau$-tangle $\alpha$,
if the associated $3$-manifold pair admits a good rectangle $R$. Then $\alpha$ is rational. Without loss of generality, it may be assumed it is rational at $x$. Then, in addition,
\begin{enumerate}[label=(\roman*)]
\item $R$ is an $\{xy,xz\}$-, $\{xx,xy\}$- or $\{xx,xz\}$-rectangle; see Figs.\ \ref{fig:xyxz}, \ref{fig:xxxy}.
\item $R$ is the former if and only if $\alpha$ is $\pm\frac{1}{n}$-rational with $n\neq \pm 1$ and odd; $R$ is the latter two if and only if 
$\alpha$ is $\pm\frac{1}{3}$-rational. 
\end{enumerate}      
\end{lemma}

Consider the $3$-manifold pair $(\Split{V_\alpha},P\cup F)$, and  
let $s,t,u,v\in\{a,b,c\}$.
Then a $(st,uv)$-rectangle is 
a good rectangle $R$ with $R\cap S_+$ joining $s_+,t_+$ and $R\cap S_-$ joining $u_-,v_-$. A $\pm\{st,uv\}$-rectangle is a good rectangle
$R$ that meets $S_\pm$ at an arc joining $s_\pm,t_\pm$ and an arc 
joining $u_\pm,v_\pm$. 

In view of Lemma \ref{lm:S_disjoint_annulus}, up to isotopy, $S$ cuts every essential annuli and M\"obius band in $\Compl{V_\alpha}$ into some good rectangles. 
For instance, the M\"obius band $M_\ast\subset \Compl{V_\alpha}$ in Fig.\ \ref{fig:mobius} is cut by $S$ into a $(cc,cc)$-rectangle $Q$; see Fig.\ \ref{fig:split_manifold}. 
 
\begin{lemma}\label{lm:disjoint_rectangles}
Suppose $R$ is good rectangle disjoint from $Q$ in the $3$-manifold pair $(\Split{V_\alpha},P\cup F)$. Then $R$ is a $(bc,cc)$-, $(cc,ac)$-, $(bc,ac)$-, $(cc,cc)$-, $(ac,bc)$, $+\{bc,cc\}$-, or $-\{ac,cc\}$-
rectangle.
\end{lemma} 
\begin{proof}
We divide it into two cases:
\subsection*{Case $1$: $R$ meets both $S_+$ and $S_-$.}
In other words, $R$ is a $(st,uv)$-rectangle. 
Observe first, by the connectedness of $R$, $a\in\{s,t\}$ if and only if $b\in\{u,v\}$. 

Suppose $s=a,u=b$. Then either $t\in\{b,c\},v\in\{a,c\}$ 
or $t=a,v=b$, and hence there are five possibilities.
The rectangle $R$ being disjoint from $Q$ and being good 
implies that $t\neq a$ or $b$ and $u\neq a$ or $b$, and hence only $(ac,bc)$-rectangle may occur.  

Suppose $a\not\in\{s,t\},b\not\in\{u,v\}$. 
Then there are $9$ possibilities, yet $R$ being disjoint from $Q$ rules out any configuration with $st=bb$ or $uv=aa$. Thus only $(cc,ac)$-, $(cc,cc)$-, $(bc,ac)$-, and $(bc,cc)$-rectangle may occur.

\subsection*{Case $2$: $R$ meets only one of $S_+,S_-$.} 
Suppose $R$ only meets $S_+$, and $R$ is a $+\{st,uv\}$-rectangle. 
Then by the connectedness of $R$, we have $a\not\in\{s,t,u,v\}$, and therefore, there are six possibilities. The rectangle $R$ being disjoint from $Q$
implies $bb\not\in\{st,uv\}$, and therefore only 
$+\{bc,bc\}$-, $+\{cc,cc\}$-, or $+\{bc,cc\}$-rectangle
may occur. The first implies there exists a disk with a half integral boundary slope in a knot exterior, an impossibility. The second implies $R$ is separating, and since $R\cap S_+$ are two parallel arcs whose boundary is in $c_+$. 
The loops $a_+,b_+$ are in the same component $X$ cut off by $R$ in $\Split {V_\alpha}$, and hence $a_-,b_+,c_-$ are all in $X$. In particular, $P\cap \partial R$ consists of two inessential arcs, contradicting $R$ being good. As a result, $R$ can only be a $+\{bc,cc\}$-rectangle.
Similarly, if $R$ meets only $S_-$, then $R$ is a $-\{ac,cc\}$-rectangle.   
\end{proof}
 
Observe that $\partial Q$ cuts $P$ into a pair of pants $P_\alpha$ and a disk, 
and $Q$ cuts $\Split{V_\alpha}$ into two components with one containing $P_\alpha$, denoted by $E_\alpha$. Furthermore, there is a canonical orientation-preserving homeomorphism $f_\alpha$ from the $3$-manifold pair $(E_\alpha,P_\alpha)$ to the $3$-manifold pair $(\Compl\alpha,F_\alpha)$ associated to the $\tau$-tangle $\alpha$ sending $b_+,a_-$ to 
$y,z$. In particular, every good rectangle in $(\Split{V_\alpha},P\cup F)$ 
disjoint from $F$ induces a good rectangle in 
the $3$-manifold pair $(\Compl{\alpha},F_\alpha)$ and vice versa. 

\begin{figure}[t]
	\begin{subfigure}{.32\linewidth}
		\centering
		\begin{overpic}[scale=.15,percent]{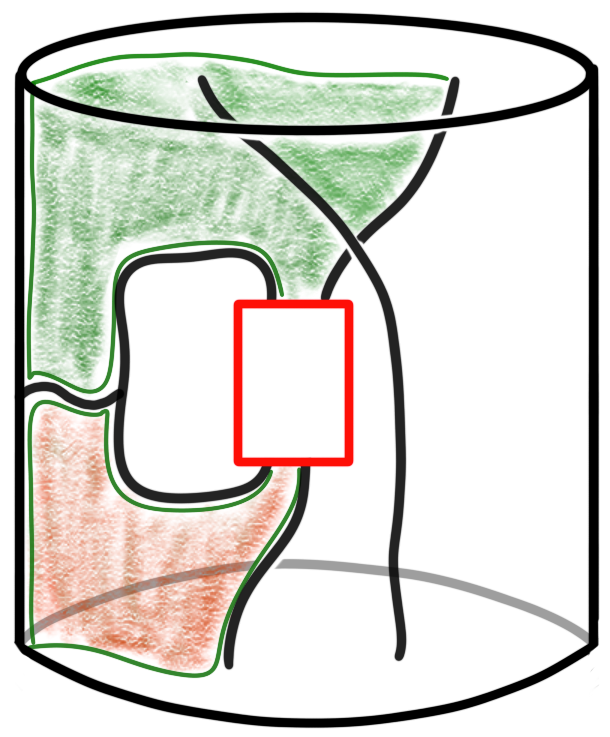}
			\put(38,45){$n$}
			\put(33,9){$a_-$}
			\put(63.8,87.5){\small $b_+$}
			\put(78,95){$c_+$}
			\put(56,9){$b_-$}
			\put(4.3,2.1){$c_-$}
			\put(17,88){\small $a_+$} 
		\end{overpic}
		\caption{$(bc,ac)$-rectangle.}
		\label{fig:parallelism_integral}
	\end{subfigure}
	\begin{subfigure}{.32\linewidth}
		\centering
		\begin{overpic}[scale=.15,percent]{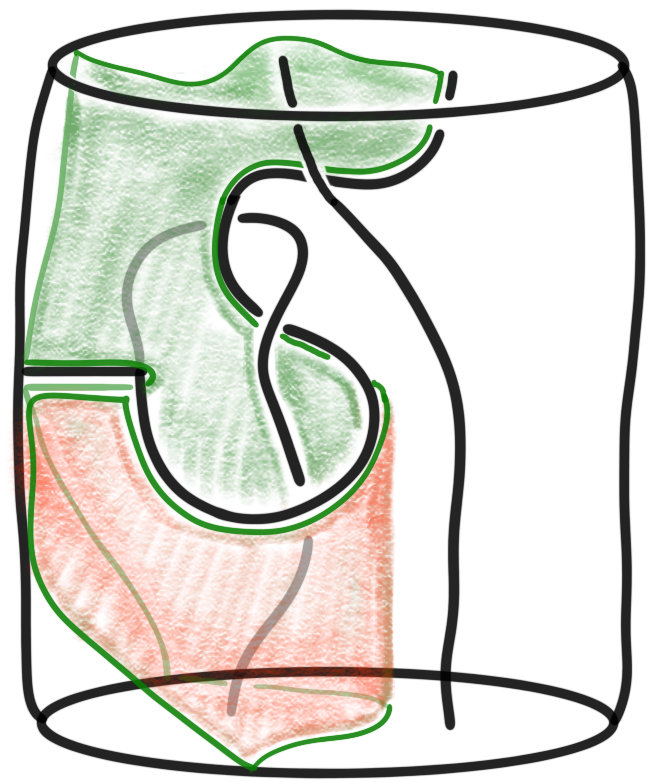}
			\put(38,89){\small $a_+$}
			\put(60,88.5){\small $b_+$}
			\put(79,94){$c_+$}
			\put(33.5,8.5){\transparent{.7}{$a_-$}}
			\put(1.1,1.5){$c_-$}
			\put(60,5){\small $b_-$}
		\end{overpic}
		\caption{$(bc,cc)$-rectangle.}
		\label{fig:parallelism_1}
	\end{subfigure} 
	\begin{subfigure}{.32\linewidth}
		\centering
		\begin{overpic}[scale=.15,percent]{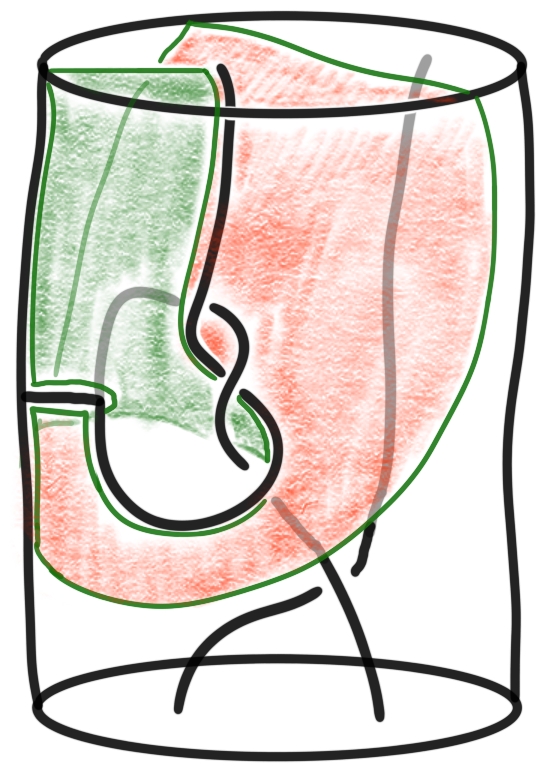}
			\put(26,6.7){\small $a_-$}
			\put(32.5,88){\small $b_+$}
			\put(57.4,90.1){\small $a_+$}
			\put(69,93){$c_+$}
			\put(52,6.1){\small $b_-$}
			\put(2,1.8){$c_-$}
		\end{overpic}
		\caption{+$\{bc,cc\}$-rectangle.}
		\label{fig:parallelism_2}
	\end{subfigure} 
	\caption{Good rectangles in $(\Split{V_\alpha},P\cup F)$.}
\end{figure}

\begin{lemma}\label{lm:rectangles_split_manifold}\hfill 
\begin{enumerate}[label=(\roman*)]
\item\label{itm:type_I} A $(bc,ac)$-rectangle exists if and only if $\alpha$ is $\pm \frac{1}{n}$-rational.
\item\label{itm:type_II} A $(bc,cc)$-, $+\{bc,cc\}$-, $(cc,ac)$-, or $-\{ac,cc\}$-rectangles exists if and only if $\alpha$ is $\pm \frac{1}{3}$-rational. 
\end{enumerate} 
\end{lemma} 
\begin{proof}
 
\ref{itm:type_I} follows from Lemma \ref{lm:rectangles_tau_tangle} and that, via $f_\alpha$, a $(xy,xz)$-rectangle in $(\Compl\alpha,F_\alpha)$ induces a 
$(bc,ac)$-rectangle in $(\Split{V_\alpha},P\cup F)$ and vice versa; see Fig.\ \ref{fig:parallelism_integral}.


Similarly, \ref{itm:type_II} follows from Lemma \ref{lm:rectangles_tau_tangle} and the fact that, via $f_\alpha$, every $(xx,xy)$-rectangle (resp.\ $(xx,xz)$-rectangle) in $(\Compl\alpha,F_\alpha)$ induces 
a $(bc,cc)$- or $+\{bc,cc\}$-rectangle (resp.\ $(cc,ac)$- or $-\{cc,ac\}$) in $(\Split{V_\alpha},P\cup F)$, depending on whether it meets $S_-$ (resp.\ $S_+$), and vice versa; see Figs.\ \ref{fig:parallelism_1}, \ref{fig:parallelism_2}.
\end{proof}

\cout{
\begin{theorem}[{Classification theorem}]\label{teo:annulus_classification} Up to isotopy, the exterior of $V_\tau$ admits
\begin{enumerate}[label=(\roman*)] 
\item  infinitely many essential annuli if $(B,\tau)$ is $\pm\frac{1}{3}$-rational.
\item  two essential annuli if $(B,\tau)$ is $\frac{1}{n}$-rational with $\vert n\vert >3$ odd.
\item   a unique essential annulus if $(B,\tau)$ is not $\frac{1}{n}$-rational.
\end{enumerate}
\end{theorem}
}

Given a non-trivial handlebody-knot $V$, let
$A\subset\Compl V$ be a separating annulus that cuts off from $\Compl V$ a solid torus $U$. Then the \emph{slope} of $A$ is defined to be the slope of the core of $A$ with respect to $U$. For instance, the frontier $A_\ast$ of a regular neighborhood of the M\"obius band $M_\ast$ in Fig.\ \ref{fig:mobius} is a separating annulus with a slope of $-\frac{1}{2}$.

\begin{lemma}\label{lm:annuli} Up to isotopy, the exterior of $V_\alpha$ admits, 
\begin{enumerate}[label=(\roman*)] 
\item\label{itm:infinite} infinitely many essential annuli if $\alpha$ is $\pm\frac{1}{3}$-rational,
\item\label{itm:two} two essential annuli $A_\ast,A_n$ if $\alpha$ is $\frac{1}{n}$-rational with $\vert n\vert>3$ odd, where $A_n$ is the frontier of $M_n$ in Fig.\ \ref{fig:Aprime} and has a slope of $\frac{n-4}{2}$, or
\item\label{itm:one} one essential annulus $A_\ast$ otherwise. 
\end{enumerate}
In addition, if $\alpha$ is $\frac{1}{3}$- or $-\frac{1}{3}$-rational, 
then $V_\alpha$ is, respectively, $\fivetwo$, $\sixthirteen$ in the handlebody-knot table \cite{IshKisMorSuz:12}. 
\end{lemma}

\begin{proof} 
Case \ref{itm:infinite}: By Figs.\ \ref{fig:V_threeone}, \ref{fig:V_neg_threeone},
if $\alpha$ is $\frac{1}{3}$-rational or $-\frac{1}{3}$-rational, then $V_\alpha$ is $\fivetwo$ and $\sixtwelve$ in \cite{IshKisMorSuz:12}. In either case, the JSJ-decomposition of the exterior $\Compl{V_\alpha}$ admits an admissible I-bundle over a once-punctured Klein bottle \cite{Wan:24i} (see also \cite{KodOzaWan:25i}). Hence the exterior of $V_\alpha$ admits infinitely many essential annuli, up to isotopy.
\begin{figure}[b]
	\begin{subfigure}{.35\linewidth}
		\centering
		\begin{overpic}[scale=.14,percent]{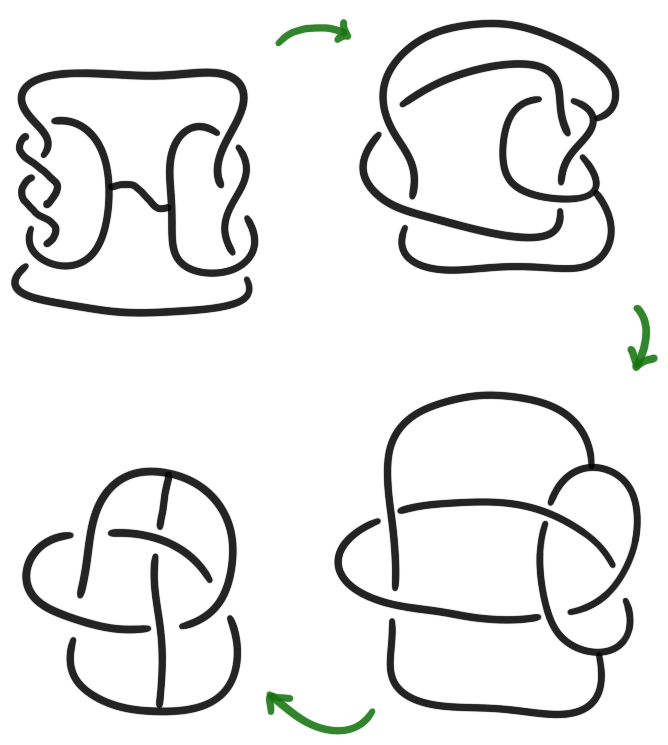}
			\put(0,2){$\fivetwo$}
		\end{overpic}
		\caption{$\frac{1}{3}$-rational $\alpha$; $\fivetwo$.}
		\label{fig:V_threeone}
	\end{subfigure}
	\begin{subfigure}{.38\linewidth}
		\centering
		\begin{overpic}[scale=.14,percent]{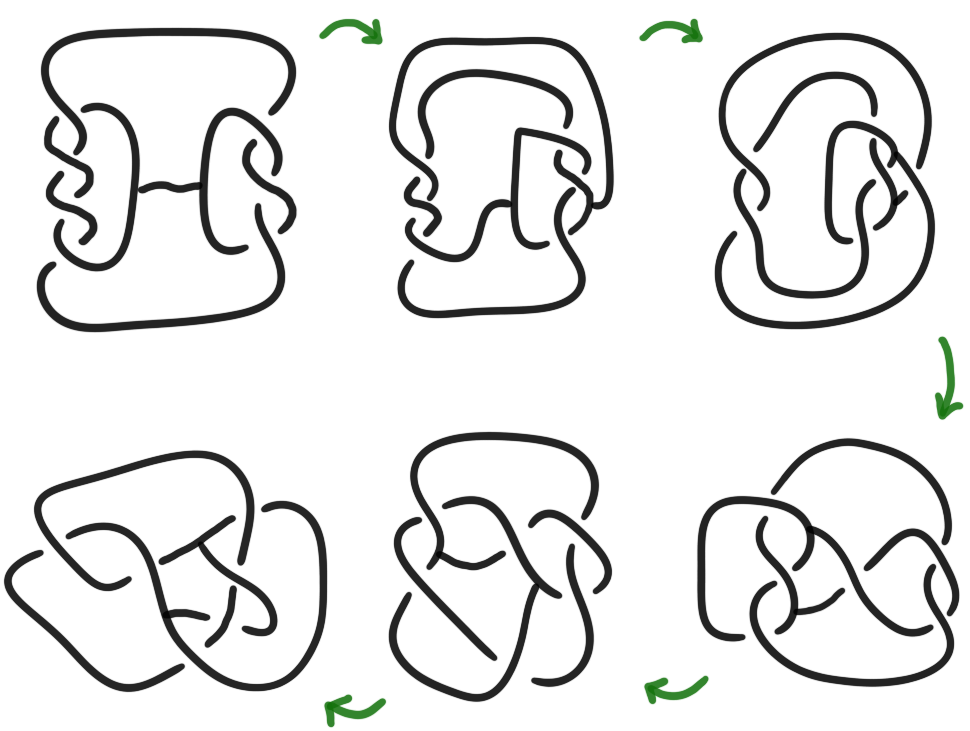}
			\put(0,1){$\sixthirteen$}
		\end{overpic}
		\caption{$-\frac{1}{3}$-rational $\alpha$, and $\sixthirteen$.}
		\label{fig:V_neg_threeone}
	\end{subfigure} 
	\begin{subfigure}{.25\linewidth}
	\centering
	\begin{overpic}[scale=.12,percent]{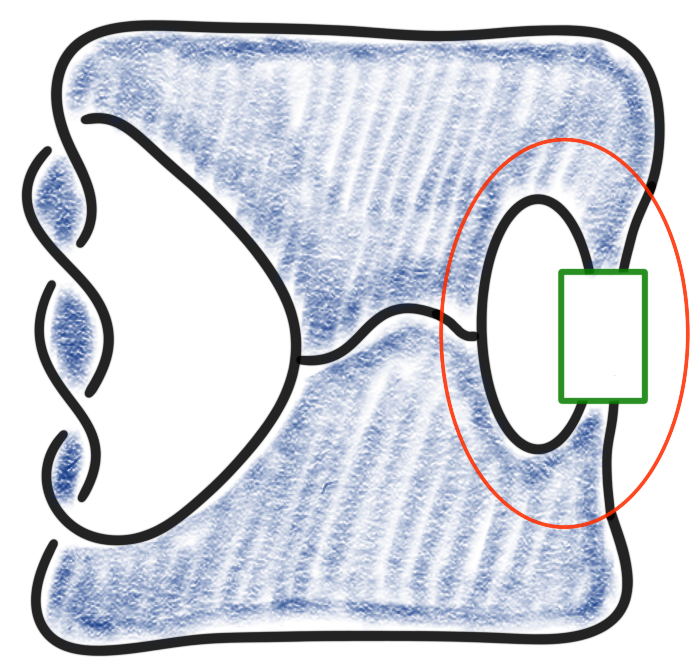}
		\put(83,43){$n$}
	\end{overpic}
	\caption{M\"obius band $M_n$.}
	\label{fig:Aprime}
\end{subfigure} 
	\caption{}
\end{figure}

Case \ref{itm:two}: By Lemma \ref{lm:rectangles_split_manifold}, $(bc,ac)$-rectangles and $(ac,bc)$-rectangles exist. The union of a $(bc,ac)$-rectangle and 
a $(ac,bc)$-rectangle induces a M\"obius band $M_n$ non-isotopic and disjoint from $M_0$; see Fig.\ \ref{fig:Aprime}. Denote by $A_n$ the frontier of a regular neighborhood of $M_n$. 
Since $M_n$ is a M\"obius band with $n-4$ half twists, the slope of $A_n$ is $\frac{n-4}{2}$.
Consider the following claim. 

\textbf{Claim: If $A_\ast$, $A_n$ are not the only essential annuli in $\Compl{V_\alpha}$, up to isotopy, then there is an essential annulus $A$ disjoint from and not isotopic to $A_\ast,A_n$.}
 
Suppose $A_\ast,A_n$ are characteristic annuli of $\Compl{V_\alpha}$. Then every annulus can be isotoped away from $A_\ast\cup A_n$, and thus the claim.
If at least one of $A_\ast,A_n$ is not 
characteristic, then by \cite{Wan:24ii} (see also \cite{KodOzaWan:25i}), 
one and exactly one of the following happens: 
the JSJ-decomposition of $\Compl{V_\alpha}$ admits 
\begin{itemize}
\item\label{itm:typeS} an admissible Seifert fibered solid torus that meets $\partial V_\tau$ at two components, 
\item\label{itm:typeM} an admissible I-bundle over a once-punctured M\"obius band, or 
\item\label{itm:typeK} an admissible I-bundle over a once-punctured Klein bottle.
\end{itemize}
In the first two cases, $\Compl{V_\alpha}$ admits at least two characteristic annuli and hence, there always exists an essential annulus not isotopic to $A_\ast$ and $A_n$ and disjoint from them. In the third case, the unique characteristic annulus $A_c$ cannot be the frontier of a regular neighborhood of some M\"obius band since its slope is never a half integer by \cite{KodOzaWan:25i}. Take $A=A_c$. Then the claim follows. 

To prove \ref{itm:two}, we suppose otherwise 
and $A$ is an essential annuli disjoint from and not isotopic to $A_\ast$ or $A_n$. 
By Lemmas \ref{lm:S_disjoint_annulus} and \ref{lm:rectangles_split_manifold}, $S$ cuts 
$A$ into some good rectangles, which are disjoint from $Q$ and hence comprise some $(ac,bc)$-, $(bc,ac)$-, or $(cc,cc)$-rectangles by Lemma \ref{lm:disjoint_rectangles}. This implies $A$ is isotopic either to $A_\ast$ or to $A_n$, a contradiction.

Case \ref{itm:one}:
Suppose otherwise. Then by the JSJ-decomposition of $\Compl{V_\alpha}$, there exists an essential annulus $A$ disjoint from and not isotopic to $A_\ast$, and by Lemma \ref{lm:S_disjoint_annulus}, $S$ cuts $A$ into some good rectangles in $(\Split{V_\alpha},P\cup F)$, which are disjoint from $Q$. Since $\alpha$ is not $\frac{1}{n}$-rational, 
%
none of the five rectangles: $(bc,ac)$-, $(bc,cc)$-, $+\{bc,cc\}$-, $(cc,ac)$-, or $-\{ac,cc\}$-rectangle can exist.
On the other hand, if an $(ac,bc)$-rectangle exists, then 
one of the above five rectangles must occur. By Lemma \ref{lm:disjoint_rectangles}, $A$ is cut into some $(cc,cc)$-rectangles and hence comprises two $(cc,cc)$-rectangles, contradicting $A$ not isotopic to $A_\ast$.  
\end{proof}

\begin{remark}
Lemma \ref{lm:annuli} is a special case of \cite[Theorem $6.2$]{OzaWan:26}, yet we cannot use the result directly from \cite{OzaWan:26} as the different techniques developed here is needed in Sections \ref{subsec:M_surface}, \ref{subsec:proof}, \ref{subsec:extension}.  
\end{remark}

\subsection{$M$-surface}\label{subsec:M_surface}
To classify $V_\alpha$, we also need to study essential surfaces with negative Euler characteristic. 
By Lemma \ref{lm:annuli}, if $\alpha$ is not $\pm\frac{1}{3}$-rational, then $A_\ast$, the frontier of a regular neighborhood of $M_\ast$, is the unique essential annulus with a slope of $-\frac{1}{2}$ in $\Compl{V_\alpha}$; see Fig.\ \ref{fig:mobius}.
By \cite[Lemma $2.1$]{FunKod:20}, there exists, up to isotopy, a unique non-separating disk $D_\ast\subset V_\alpha$ disjoint from $\partial A_\ast$. This allows the following definition.

\begin{definition}\label{def:M_surface} 
Given a $\tau$-tangle $\alpha$ which is $\pm\frac{1}{3}$-rational, an $D_\ast$-surface of $V_\alpha$ is an essential pair of pants $G$ whose two boundary components are isotopic to $\partial D_\ast$ with the other boundary component not meridional.	
\end{definition} 
The exterior $\Compl{V_\alpha}$ contains at least one $D_\ast$-surface, namely, $S$ in Fig.\ \ref{fig:fourone_tau}, where the components $a,b\subset \partial S$ are isotopic to $\partial D_\ast$. Let $D_a,D_b\subset V_\alpha$ 
be the disks bounded by $a,b$, and $\hat{S}$
the union $S\cup D_a\cup D_b$.
 
 \begin{lemma}\label{lm:unique_M_surface}
 Suppose $\alpha$ is not $\pm\frac{1}{3}$-rational.	
 Then $S$ in Fig.\ \ref{fig:fourone_tau} is the unique $D_\ast$-surface, up to isotopy.
 \end{lemma}
 \begin{proof}
 Suppose $S'$ is another $D_\ast$-surface, and denote by $a',b',c'$ the three components of $\partial S'$ with $c'$ the one not isotopic to $\partial D_\ast$.
 Isotope $S'$ so $\vert S\cap S'\vert$ 
 is minimized. By the essentiality, components of $S\cap S'$ are essential arcs or circles in 
 $S,S'$.
 	
\subsection*{Case $1$: $S\cap S'$ contains a circle}
Let $\gamma\subset S\cap S'$ be an outermost circle in $S'$, and $A',A$ be the annuli cut off by $\gamma$ from $S',S$, respectively. 
Note the union $A\cup A'$ is an incompressible annulus in $\Compl{V_\alpha}$. 
Since $A\cup A'$ can be isotoped away from $S$, by Lemma \ref{lm:S_disjoint_annulus}, $A\cup A'$ is $\partial$-compressible. Hence by the irreducibility of $V_\alpha$ (see Lemma \ref{lm:non_trivial_essen_irreducible}), 
$A\cup A'$ is $\partial$-parallel. Particularly,
$A\cup A'$ cuts off from $\Compl{V_\alpha}$ a solid torus $W$ where the cores of $A,A'$ are longitudes; see Fig.\ \ref{fig:no_circle_c_cprime} for the case where $\partial(A\cup A')=c\cup c'$. Thus one can isotope $S'$ 
through $W$ and remove the intersection $\gamma$, a contradiction.

\subsection*{Case $2$: $S\cap S'$ contains no arcs.}
Since $S,S'$ are both $M$-surfaces, by the minimality, no arc in $S\cap S'$ meets the components of $\partial S,\partial S'$ parallel to $\partial D_\ast$. 
Therefore all arcs $S\cap S'$ in $S$ (resp.\ in $S'$) are separating and parallel. 
Denote by $A_a,A_b\subset S'$ the annular components cut off by $S\cap S'$ with 
$a'\subset \partial A_a,b'\subset \partial A_b$. 
Let $c_i$ be the intersection $c'\cap A_i$, $i=a,b$; see Fig.\ \ref{fig:intersection_M_surfaces}. It may be assumed that $A_a$ meets $S_+$ and $A_b$ meets $S_-$. Since $a',b'$ each bound a disk in $V_\alpha$, $A_a,A_b$ are separating in $\Split{V_\alpha}$. In addition, because $a',b'$ are parallel to $\partial D_\ast$, $c_a$ cuts off from $P$ either a disk or a $3$-punctured sphere. The former contradicts minimality, while the latter implies that $c_b$ cuts off a disk from $P$, contradicting the minimality; see Figs.\ \ref{fig:c_one_inessential}, \ref{fig:c_two_inessential}. 
 \begin{figure}[t]
 \begin{subfigure}{.32\linewidth}
 			\centering
 			\begin{overpic}[scale=.13,percent]{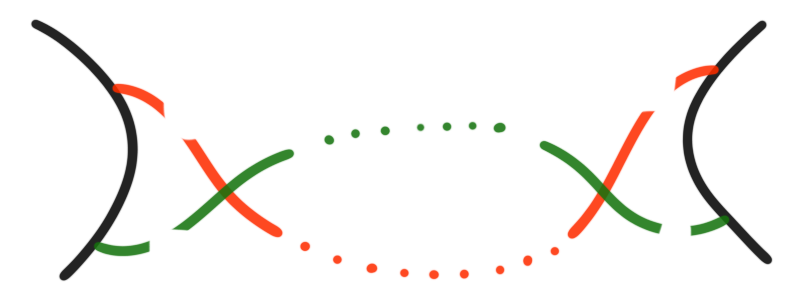}
 				\put(20,21){\footnotesize $A'$}
 				\put(77,24){\footnotesize $A'$}
 				\put(18,4){\footnotesize $A$}
 				\put(81,4){\footnotesize $A$}
 				\put(6,6){\footnotesize $c$}
 				\put(7,21){\footnotesize $c'$} 
 				\put(0,13){\footnotesize $V_\tau$} 
 				\put(93,15){\footnotesize $V_\tau$} 
 			\end{overpic}
 			\caption{$\partial (A\cup A')=c\cup c'$.}
 			\label{fig:no_circle_c_cprime}
 \end{subfigure}
 \begin{subfigure}{.32\linewidth}
 			\centering
 			\begin{overpic}[scale=.13,percent]{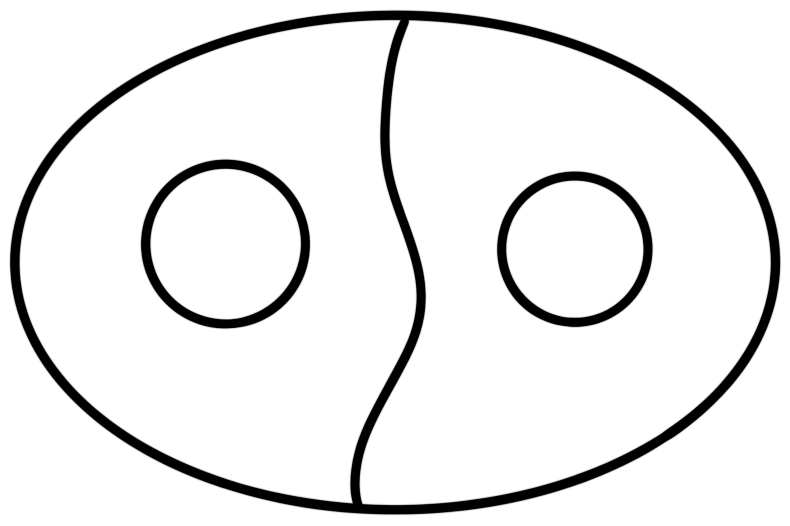}
 				\put(37,43){$a'$}
 				\put(80,43){$b'$}
 				\put(7,27){$A_a$}
 				\put(70,15){$A_b$}
 				\put(88,55){$c_b$}
 				\put(10,59){$c_a$}
 			\end{overpic}
 			\caption{Intersection $S'\cap S\subset S'$.}
 			\label{fig:intersection_M_surfaces}
 \end{subfigure}
 \begin{subfigure}{.32\linewidth}
 			\centering
 			\begin{overpic}[scale=.11,percent]{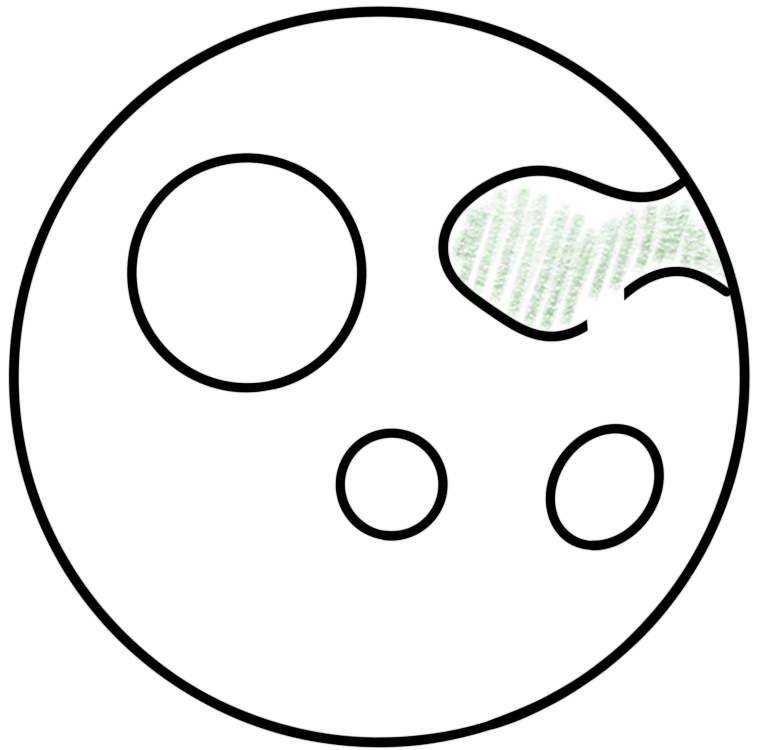}
 				\put(77.5,53.5){$c_a$}
 				\put(80,89){$c_+$}
 				\put(40,77.5){$c_-$}
 				\put(50,21){$a_-$}
 				\put(70,20){$b_+$}	
 			\end{overpic}
 			\caption{Inessential $c_a\subset P$.}
 			\label{fig:c_one_inessential}
 \end{subfigure}  	
 \begin{subfigure}{.32\linewidth}
 			\centering
 			\begin{overpic}[scale=.11,percent]{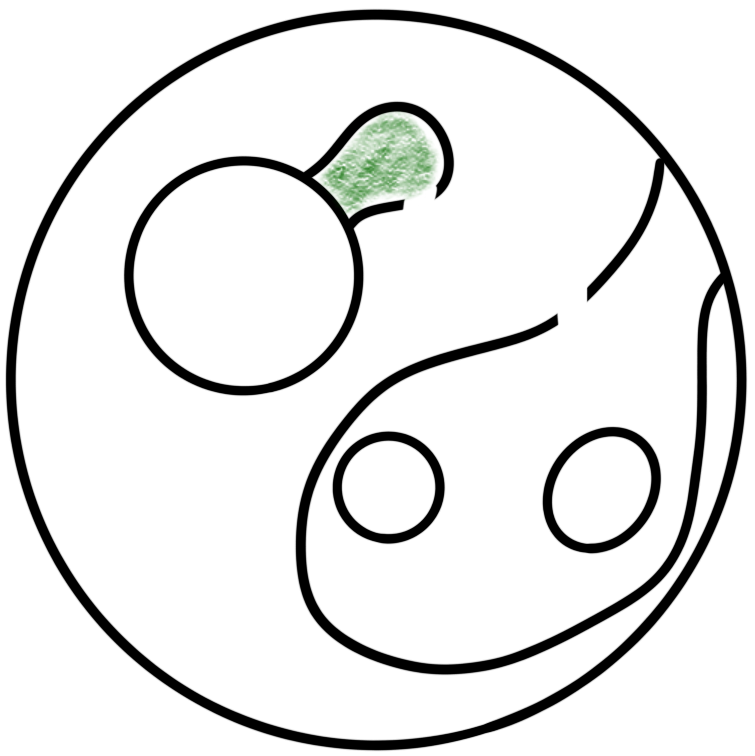}		
 				\put(74,54){$c_a$}
 				\put(53.5,68){$c_b$}
 				\put(80,89){$c_+$}
 				\put(29,80){$c_-$}
 				\put(48,22){$a_-$}
 				\put(69,22){$b_+$}
 			\end{overpic}
 			\caption{Essential $c_a\subset P$.}
 			\label{fig:c_two_inessential}
 \end{subfigure}  	
 \begin{subfigure}{.45\linewidth}
 			\centering
 			\begin{overpic}[scale=.11,percent]{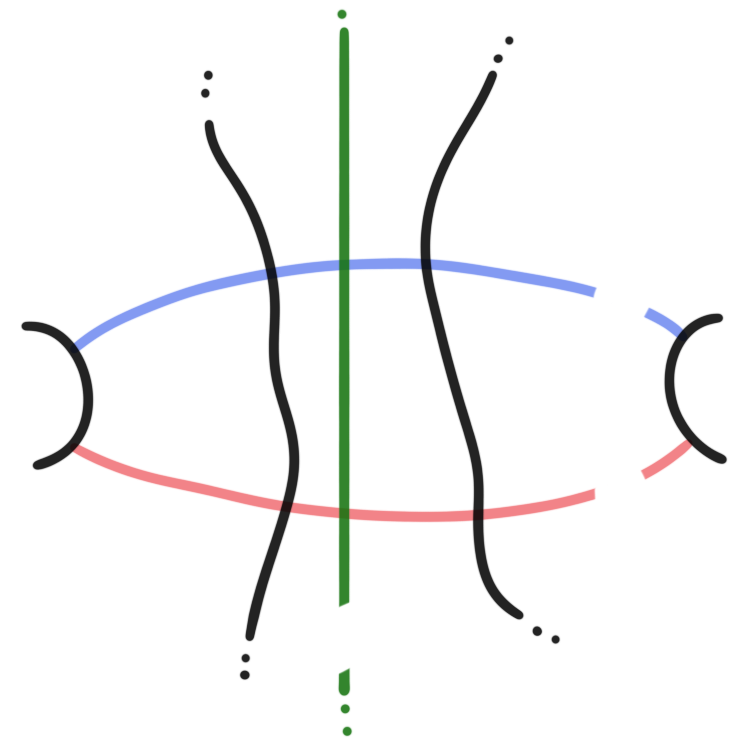}
 				\put(-2,43){\footnotesize $V_\alpha$}
 				\put(94,43){\footnotesize $V_\alpha$}
 				\put(80,59){\footnotesize $\hat S'$}
 				\put(80,29){\footnotesize $\hat S$}
 				\put(43,12.5){\footnotesize $Q$}
 				\put(14,85){\footnotesize $\Gamma_\alpha$}
 				\put(65,8){\footnotesize $\Gamma_\alpha$}
 				\put(70,45){\footnotesize $B$}
 			\end{overpic}
 			\caption{Disks $D,D'$; tangle $(B,T)$.}
 			\label{fig:S_Sprime_paralell}
 \end{subfigure}
 \caption{}
 \end{figure}    
 	
 	As a result, $S',S$ are disjoint.
 	Let $\hat S'$ be the union of $S'$ and the disks bounded by $a',b'$. It may be assumed that $\hat S, \hat S'$ are also disjoint, and hence $\hat S, \hat S'$ are parallel disks in the exterior of $U=V_\alpha-\openrnbhd{D_\ast}$. Particularly, $\hat S\cup \hat S'$ separates $\Compl U$.
 	This implies $S\cup S'$ also separates $\Compl{V_\alpha}$, and $c\cup c'$ 
 	cuts off an annulus $A$ from $\partial V_\alpha$. The union $\hat S\cup \hat S'\cup A$ bounds a $3$-ball $B$ (see Fig.\ \ref{fig:S_Sprime_paralell}) such that $B$ meets 
 	the spine $\Gamma_\alpha$ at a two-string tangle $T$, Up to isotopy, the $(cc,cc)$-rectangle $Q$ meets $B$ at a disk separating the two strings in $T$. By the atoroidality, each string is trivial. This implies $S,S'$ are parallel.
 \end{proof}
  
\begin{lemma}\label{lm:equivalence}
Suppose $\alpha,\beta$ are not $\pm\frac{1}{3}$-or $\frac{1}{5}$-rational. Then 
\begin{enumerate}[label=(\roman*)]
\item\label{itm:orientation_preserving} Every homeomorphism $f:(\sphere,V_\alpha)\rightarrow (\sphere,V_\beta)$ is orientation-preserving.
\item\label{itm:restriction} Every equivalence between $V_\alpha,V_\beta$ induces an $x$-equivalence between $\alpha,\beta$. 
\end{enumerate} 
\end{lemma}
\begin{proof}
By Lemma \ref{lm:annuli}, the frontier $A_\ast$ of $M_\ast$ is the unique essential separating annulus with a slope of $\pm\frac{1}{2}$ in the exterior $\Compl{V_\alpha}$ (resp.\ $\Compl{V_\beta}$). Therefore it may be assumed that $f(A_\ast)= A_\ast, f(M_\ast)=M_\ast$. Since both annuli are of slope $\frac{1}{2}$, $f$ is orientation-preserving.
This proves \ref{itm:orientation_preserving}.

By Lemma \ref{lm:unique_M_surface}, 
we may assume that $f(S)=S$. 
Recall that $S\cup M_\ast$ cuts off from $\Compl{V_\alpha}$ a $3$-manifold pair 
$(E_\alpha,P_\alpha)$ 
(resp.\ $(E_\beta,P_\beta)$),
and therefore $f$ induces an equivalence
$f':(E_\alpha,P_\alpha)\rightarrow (E_\beta,P_\beta)$ preserving $b_+\cup a_-$. 
Recall also there is an equivalence $f_\tau:(E_\tau,P_\tau)\rightarrow (\Compl{\tau},F_\tau)$ sending $b_+\cup a_-$ to $y\cup z$, $\tau=\alpha,\beta$. 
Extending the composition 
\[g:=g_\beta\circ f''\circ g^{-1}_\alpha: (E(\alpha),F_\alpha)\rightarrow (E(\beta),F_\beta).\]
to an $x$-equivalence between $\alpha$ and $\beta$ gives us \ref{itm:restriction}. 
\end{proof}

\subsection{Proof of Main Theorems}\label{subsec:proof}
 \begin{figure}[t]
 	\includegraphics[scale=.08]{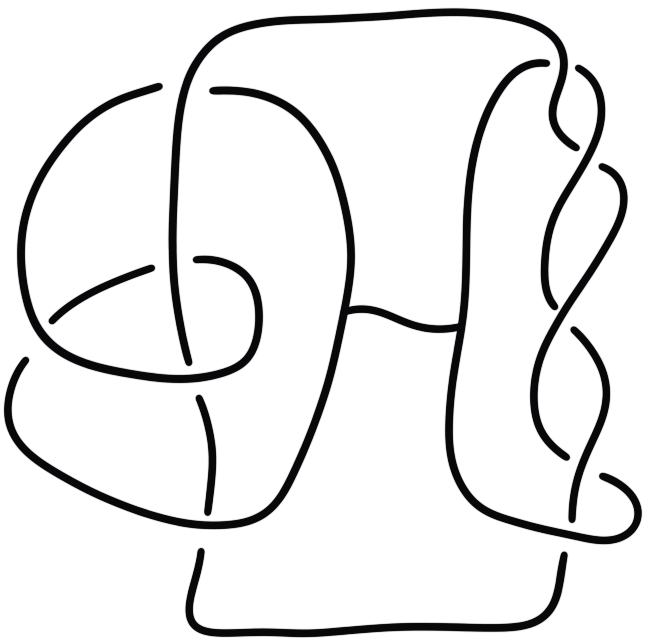}
 	\
	\includegraphics[scale=.08]{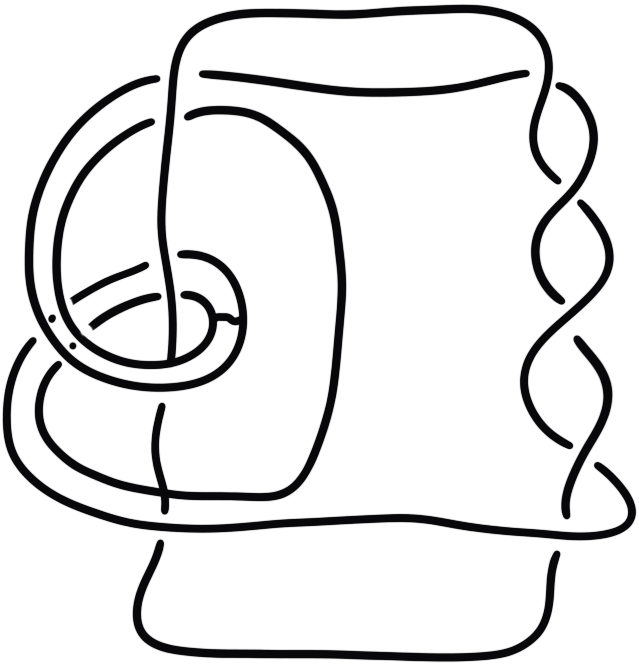}
	\
	\includegraphics[scale=.08]{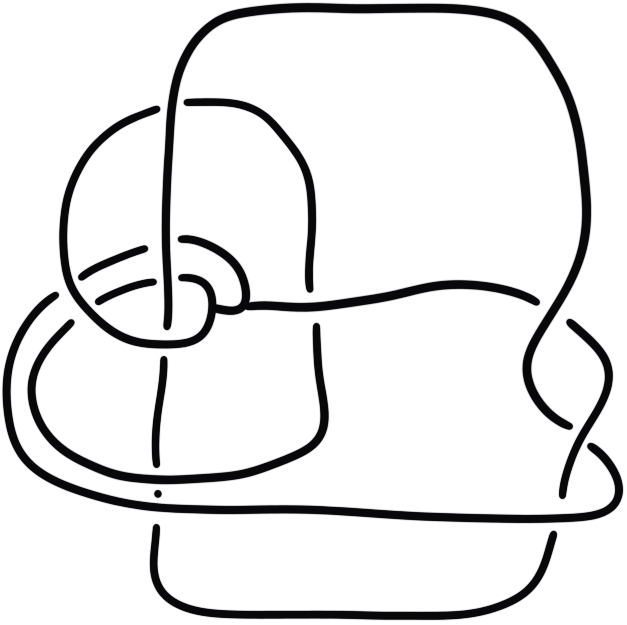}
	\
	\includegraphics[scale=.08]{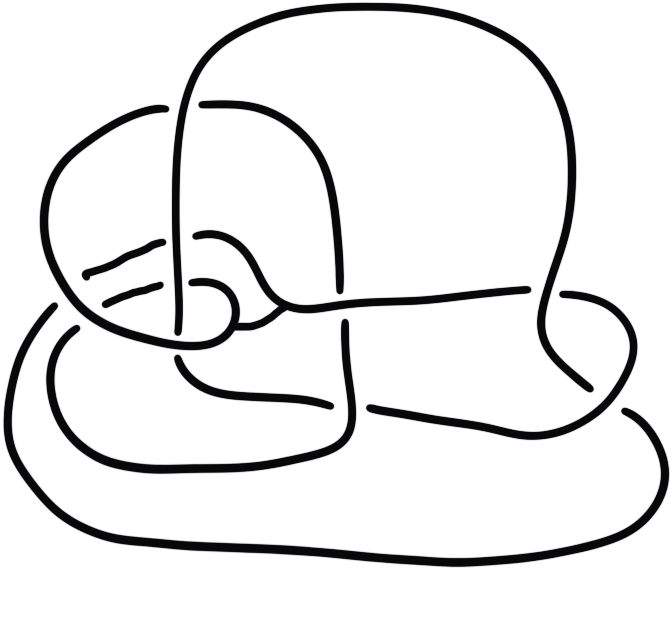}
	\
	\includegraphics[scale=.08]{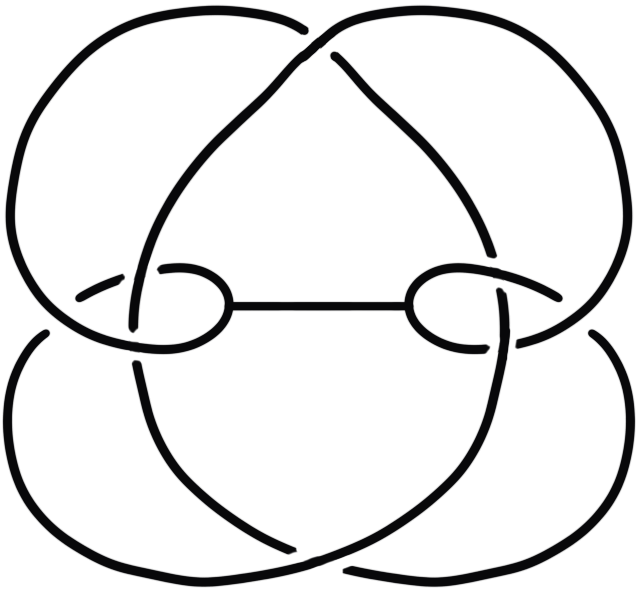}
	\
	\begin{overpic}[scale=.08,pe7cent]{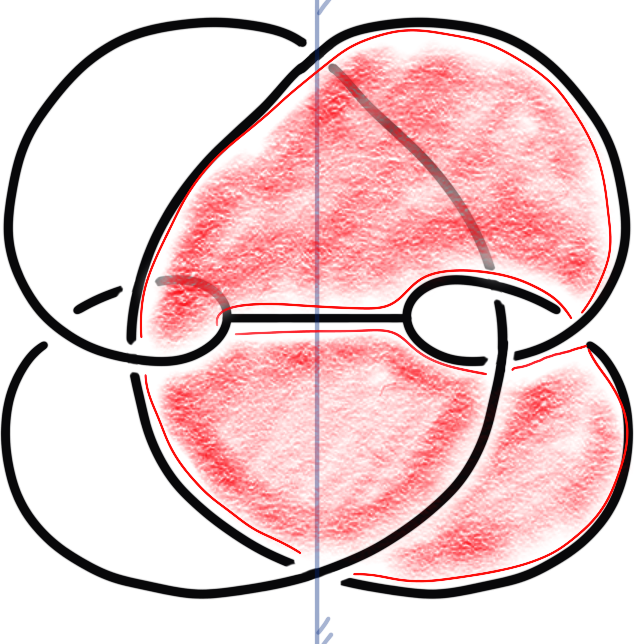}
	\put(47,62){\footnotesize $M_5$}
 	\put(30,-7){\tiny mirror}
 	\end{overpic}
	\caption{Deform $V_{\frac{1}{5}}$ into a symmetric form.}
	\label{fig:deformation_symmetric_form}
\end{figure}

If $\alpha$ is $\frac{1}{3}$-, $-\frac{1}{3}$-, or $\frac{1}{5}$-rational, then $V_\alpha$ is $\fivetwo$, $\sixthirteen$, or $\seventhirtytwo$ in the handlebody-knot table \cite{IshKisMorSuz:12,BelPaoPaoWan:25p}, respectively.  
\begin{lemma}[{\cite{LeeLee:12, Wan:23, KodOzaWan:25ii}}]\label{lm:five_six}
The handlebody-knots $\fivetwo,\sixthirteen$ are chiral, and	
\[\psym{\fivetwo}\simeq \mathbb{Z}_2\times \mathbb{Z}_2,\quad  \psym{\sixthirteen}\simeq \mathbb{Z}_2.\]
\end{lemma}
\begin{proof}
Their chirality is first determined in \cite{LeeLee:12}. The symmetry group of $\fivetwo$ are computed in \cite{Wan:23} and $\sixthirteen$ in \cite{KodOzaWan:25ii}.
\end{proof}
 
\begin{lemma}\label{lm:chirality}
The handlebody-knot $V_\alpha$ is amphichiral if and only if $\alpha$ is $\frac{1}{5}$-rational.
\end{lemma}
\begin{proof}
The ``only if'' direction follows from Lemma \ref{lm:equivalence}\ref{itm:orientation_preserving} and Lemma \ref{lm:five_six}. For the ``if'' direction, we deform $V_{\frac{1}{5}}$ into a symmetric form through the deformation in Fig.\ \ref{fig:deformation_symmetric_form}.
Then observe that the composition of a half twist of the M\"obius band $M_5$ in the last figure of Fig.\ \ref{fig:deformation_symmetric_form} and a mirror image about the $y$-axis preserves the diagram.
\end{proof}

\begin{proof}[Proof of Theorem \ref{teo:classification}]
Suppose $\alpha,\beta$ are two $x$-inequivalent $\tau$-tangles.	
If one of them is $\frac{1}{3}$-, $-\frac{1}{3}$-, or $\frac{1}{5}$-rational, then by Lemma \ref{lm:annuli}, $V_\alpha$ is not equivalent to $V_\beta$ or $\overline{V_\beta}$.  
If both $\alpha,\beta$ are not $\pm\frac{1}{3}$-rational or $\frac{1}{5}$-rational.
Then Lemma \ref{lm:equivalence} implies $V_\alpha$ is not equivalent to $V_\beta$ or $\overline{V_\beta}$. 
Combining with Lemma \ref{lm:chirality} completes the proof. 
\end{proof}

\begin{lemma}[{\cite[Theorem $2$]{McC:06}}]\label{lm:dehn_twists}
Let $M$ be a compact orientable $3$-manifold, and $h\in \paut(M)$. If the restriction of $h$ on $\partial M$ is Dehn twists along some disjoint loops $c_1,\dots, c_n\subset \partial M$, then there is a collection of disjoint annuli and disks in $M$ each of whose boundary component is isotopic to some $c_i$, and $h$ is isotopic to some composition of Dehn twists about these annuli and disks.
\end{lemma} 

\begin{lemma}[{\cite[Lemma $2.3$]{ChoKod:13}}]\label{lm:trivial_mapping_class_group}
Let $M$ be an atoroidal genus two handlebody-knot exterior. Then $\mathrm{MCG}(M,\rel\ \partial M)=\Zone$.  
\end{lemma}

\begin{lemma}\label{lm:tangle_symmetry}
Suppose $\alpha$ is an atoroidal $\tau$-tangle. 
Then
$\mathrm{MCG}_+(B,\alpha,x)\leq \mathbb{Z}_2$,
and $\mathrm{MCG}_+(B,\alpha,x)\simeq \mathbb{Z}_2$ if and only if $\alpha$ is $x$-symmetric.   
\end{lemma}
\begin{proof}
Since $\mathrm{MCG}_+(B,\alpha,x)$ acts on the set $\{y,z\}$, 
there is a homomorphism
\[\Phi:\mathrm{MCG}_+(B,\alpha,x)\rightarrow \mathbb{Z}_2,\]
and $\alpha$ is $x$-symmetric if and only if $\Phi$ is onto. Therefore it suffices to 
show $\Phi$ is injective.
Consider $f\in \paut(B,\alpha,x)$ with $\Phi([f])=0\in \mathbb{Z}_2$, namely 
$f(y)=y$ and $f(z)=z$. Since the pure mapping class group of the $3$-punctured sphere $\partial B-\{x,y,z\}$ is trivial, one can isotope $f$ in $\paut(B,\alpha,x)$ so $f\vert_{\partial B}=\id$. By the uniqueness of a regular neighborhood, we may further assume $f(\rnbhd{\alpha})=\rnbhd{\alpha}$. 
Let $g$ be the restriction of $f$ on $\Compl{\alpha}$, and observe that the restriction of $g$ on the pair of pants $F:=\partial_f\rnbhd{\alpha}$ is isotopic, in $\paut(F, \mathrm{rel}\ \partial F)$, to either the identity $\id$ or some Dehn twist along components of $\partial F$. The latter is impossible by Lemma \ref{lm:dehn_twists}, so we may assume $g$ restricts to the identity $\id$ on $\partial \Compl{\alpha}$. Hence $g$ is isotopic to $\id$ in $\paut(\Compl{\alpha},\mathrm{rel}\ \partial \Compl{\alpha})$ by Lemma \ref{lm:trivial_mapping_class_group}. This implies $f$ can be isotoped in $\paut(B,\rnbhd{\alpha},\alpha)$ so $f$ restricts to the identity $\id$ on $\Compl{\alpha}$. Regard $\rnbhd{\alpha}$ as a unit $3$-ball with $\alpha$ a cone at the origin.
Then the Alexander trick gives an isotopy between $f$ to $\id$ in $\paut(B,\rnbhd{\alpha},\alpha)$. This proves $\Phi$ is injective. 
\end{proof}

We have the following as a corollary of  \cite{Hat:76,Hat:99}.
\begin{lemma}\label{lm:surface_preserving_isotopy}
Let $M$ be a compact $3$-submanifold of $\sphere$, and $S$ an incompressible surface in $M$. Given $f,g\in \mathrm{Homeo}(M,S)$, 
if $f,g$ are isotopic in $\mathrm{Homeo}(M)$ (resp.\ relative to $\partial M$),
then $f,g$ are isotopic in $\mathrm{Homeo}(M,S)$ (resp.\ relative to $\partial M$). 
\end{lemma}

\begin{lemma}\label{lm:stabilizer}
The stabilizer 
$\mathrm{MCG}_+(\sphere,V_\alpha,[M_\ast])\leq \mathrm{MCG}_+(\sphere,V_\alpha)$ of the isotopy class of $M_\ast$ is isomorphic to $\mathrm{MCG}_+(B,\alpha,x)$ if $\alpha$ is not $\pm\frac{1}{3}$-rational. 
\end{lemma}
\begin{proof}
Consider a mapping class $[f]\in\pmcg(S^3,V_\alpha,[M_\ast])$; it may be assumed that $f(M_\ast)=M_\ast$. By the construction in the proof of Lemma \ref{lm:equivalence}\ref{itm:restriction}, 
we may further assume $f(B)=B$ and $f(\alpha)=\alpha$. Particularly, $f$ restricts to a homeomorphism in $\paut(B,\alpha,x)$. 

\noindent
\textbf{Claim: Given $f,g\in \paut(S^3,V_\alpha,B)$, if $f,g$ are isotopic in $\paut(S^3,V_\alpha)$, then $f,g$ are isotopic in 
	$\paut(S^3,V_\alpha,B)$.}
		 
Observe first the restrictions $f_1,g_1$ of $f,g$ on $\Compl{V_\alpha}$ are in $\paut(\Compl{V_\alpha},S_B)$, and $f_1, g_1$ are isotopic in $\paut(\Compl{V_\alpha})$. Thus by Lemma \ref{lm:surface_preserving_isotopy}, $f_1, g_1$ are isotopic in $\paut(\Compl{V_\alpha},S_B)$. 
Therefore, $f$ is isotopic in $\paut(S^3,V_\alpha,B)$ to a homeomorphism $f'$ that restricts to $g_1$ on $\Compl{V_\alpha}$.

Denote by $H_t$ the isotopy between $f,f'$, and by $f_2, g_2\in \paut(V_\alpha, \partial B\cap V_\alpha)$ the restrictions of $f',g$ on $V_\alpha$.
In particular, $f_2=g_2$ on $\partial V_\alpha$, and also $f_2,g_2$ are isotopic in $\paut(V_\alpha)$ because $f,g$ are isotopic in $\paut(S^3,V_\alpha)$.
Given that $\paut(\partial V_\alpha)$ has contractible component by \cite{Sco:70}, $f_2, g_2$ are isotopic in $\paut(V_\alpha)$ relative to $\partial V_\alpha$. Applying Lemma \ref{lm:surface_preserving_isotopy}, we see $f_2,g_2$ are isotopic
in $\paut(V_\alpha,V_\alpha\cap \partial B)$
relative to $\partial V_\alpha$. 
The isotopy between $f_2,g_2$ can therefore be extended to an isotopy $H_t'$ between $f',g$ in $\paut(\sphere,V_\alpha,B)$ relative to $\Compl{V_\alpha}$. Combining  $H_t,H_t'$ proves the assertion. 

Since two homeomorphisms in $\paut(B,\alpha)$ having the same images of $y,z$ are isotopic in $\paut(B,\alpha)$, the claim implies the homomorphism 
\[\Psi:\mathrm{MCG}_+(S^3,V_\alpha, [M_\ast])\rightarrow \mathrm{MCG}_+(B,\alpha,x)\]
given by the restriction of $f$ on $B$ is well-defined.

We now show that $\Psi$ is injective.
Suppose the image of $[f]$ under $\Psi$ is trivial in $\mathrm{MCG}_+(B,\alpha,x)$.  
Then $f$ fixes $x,y,z$. Because the pure mapping class of a $3$-punctured sphere is trivial, it may be assumed that $f$ restricts to the identity $\id$ on $S_B=\Compl {V_\alpha}\cap \partial B$. Observe that $ V_\alpha-\mathring{B}$ 
consists of a solid torus $U$ and a $3$-ball $U'$, and the frontier $F$ of $U\cup U'\subset \Compl{B}$ is the intersection $\Compl{B}\cap \partial V_\alpha$. Since the restriction of $f$ on $\partial U$ preserves the meridian and preferred longitude of $U$, up to isotopy, if the restriction of $f$ on $F$ is not isotopic to $\id$ in $\pmcg(F ,\mathrm{rel}\ \partial F)$, then it is isotopic to some Dehn twists along $\partial F$. This implies the restriction of $f$ on $\partial \Compl{V_\alpha\cup B}=F\cup S_B$ is some Dehn twists along $F\cap S_B$, contradicting Lemma \ref{lm:dehn_twists}. In the same way, 
the restriction $f$ on $B\cap \partial V_\alpha$ 
is isotopic to $\id$. As a result, it may be assumed that $f$ restricts to $\id$ on $\partial V_\alpha \cup S_B$. By Lemma \ref{lm:trivial_mapping_class_group}, we may further assume $f$ restricts to $\id$ on 
$\Compl{V_\alpha}$. Applying the Alexander trick, one can isotope $f$, relative to $\Compl{V_\alpha}$, 
so it restricts to $\id$ on the disks in $V_\alpha\cap \partial B$. Applying Lemma \ref{lm:trivial_mapping_class_group} again, we see the restriction of $f$ on $V_\alpha$ is isotopic, relative to the union $(\partial V_\alpha)\cup (V_\alpha\cap \partial B)$, to $\id$. Consequently,
$f,\id$ are isotopic in $\mathrm{MCG}_+(S^3,V_\alpha)$, and $\Psi$ is injective.

To see the surjectivity, recall the self-homeomorphism $\pi$ of $(\sphere,\Gamma,B)$ swaps $y,z$ (see Fig.\ \ref{fig:hc_fourone}), and induces a self-homeomorphism, still denoted by $\pi$, 
of $(\Compl{B}, \Compl{B}\cap \Gamma_\alpha)$ 
that swaps $y,z$ and fixes $x$. 
Given $[f]\in \pmcg(B,\alpha,x)$, 
since the pure mapping class group of a $3$-punctured sphere is trivial, if $f$ swaps $y,z$, then it may be assumed that the restrictions of $f$ and $\pi$ agree on $\partial B$, and hence $[f]$ is the image of $[f\cup \pi]$ under $\Psi$. Similarly, if $f$ fixes $y,z$, then it may be assumed that $f$ restricts to $\id$ on $\partial B$, so $[f]$ is the image of $[f\cup \id]$ under $\Psi$. 
\end{proof}

\begin{proof}[Proof of Theorem \ref{teo:symmetry}]
Theorem \ref{teo:symmetry}\ref{itm:three}, the case $\alpha$ is $\pm\frac{1}{3}$-rational, is Lemma \ref{lm:five_six}.
	
Now, if $\alpha$ is not $\pm\frac{1}{3}$-rational, then by Lemma \ref{lm:annuli}, $M_\ast$ is the unique M\"obius band with a slope of $-\frac{1}{2}$, so $\pmcg(\sphere,V_\alpha)=\pmcg(\sphere,V_\alpha,[M_\ast])$. Thus by Lemmas \ref{lm:stabilizer} and \ref{lm:tangle_symmetry}, 
$\pmcg(\sphere,V_\alpha)\leq \mathbf{Z}_2$,
and the equality holds if and only if $\alpha$ is $x$-symmetric.
In addition, if $\alpha$ is not $\frac{1}{5}$-rational, then $V_\alpha$ is chiral by Lemma \ref{lm:chirality} and $\mcg(\sphere,V_\alpha)=\pmcg(\sphere,V_\alpha)$.
This, together with Lemma \ref{lm:five_six}, proves Theorem \ref{teo:symmetry}\ref{itm:three}, \ref{itm:x_symmetric} and \ref{itm:not_x_symmetric}. 

To complete the proof, it suffices to determine the symmetry group of $V_{\frac{1}{5}}$. By Lemma \ref{lm:chirality}, it is amphichiral. 
Therefore, there is a short exact sequence
\[0\rightarrow\mathbb{Z}_2\simeq \pmcg(\sphere,V_\alpha)\rightarrow \mcg(\sphere,V_\alpha)\rightarrow \mathbb{Z}_2.\]
Now, the generator $t$ of $\pmcg(\sphere,V_\alpha)$ 
is given by the rotation $\rho$ about the $x$-axis in the last diagram in Fig.\ \ref{fig:deformation_symmetric_form}, and $\rho$ restricts to the hyper-elliptic involution on $\partial V_\alpha$, which gives the center of $\mcg(\partial V_\alpha)$. 
Since $\mcg(\sphere,V_\alpha)\rightarrow \mcg(\partial V_\alpha)$ is injective, $t$ is in the center of $\mcg(\sphere,V_\alpha)$,
so $\mcg(\sphere,V_{\frac{1}{5}})\simeq \Ztwo\times \Ztwo$.  
\end{proof}

\subsection{Extension}\label{subsec:extension}
Consider the spatial handcuff graph $\Gamma^k$ obtained by twisting the disk $D$ bounded by a loop of 
the spatial handcuff graph $\hcfourone$; see Fig.\ \ref{fig:Gamma_k}.
Denote by $V^k_\alpha$ the associated handlebody-knot obtained by replacing $(B,B\cap \Gamma^k)$ with a 
$\tau$-tangle $\alpha$, where $B$ is the $3$-ball in Fig.\ \ref{fig:Gamma_k}; see Fig.\ \ref{fig:V_k_tau} for an example.
%
%
In particular, the exteriors $\Compl {V^k_\alpha}$ and $ \Compl{V_\alpha}$ are homeomorphic, and hence Lemmas \ref{lm:atoroidality}, \ref{lm:non_trivial_essen_irreducible}, \ref{lm:S_disjoint_annulus} generalize to the case of $V^k_\alpha$. 
In particular, $V^k_\alpha$ is atoroidal and irreducible if $\alpha$ is atoroidal and non-trivial.

\begin{figure}[h] 
	\begin{subfigure}{.45\linewidth}
		\centering
		\begin{overpic}[scale=.12,percent]{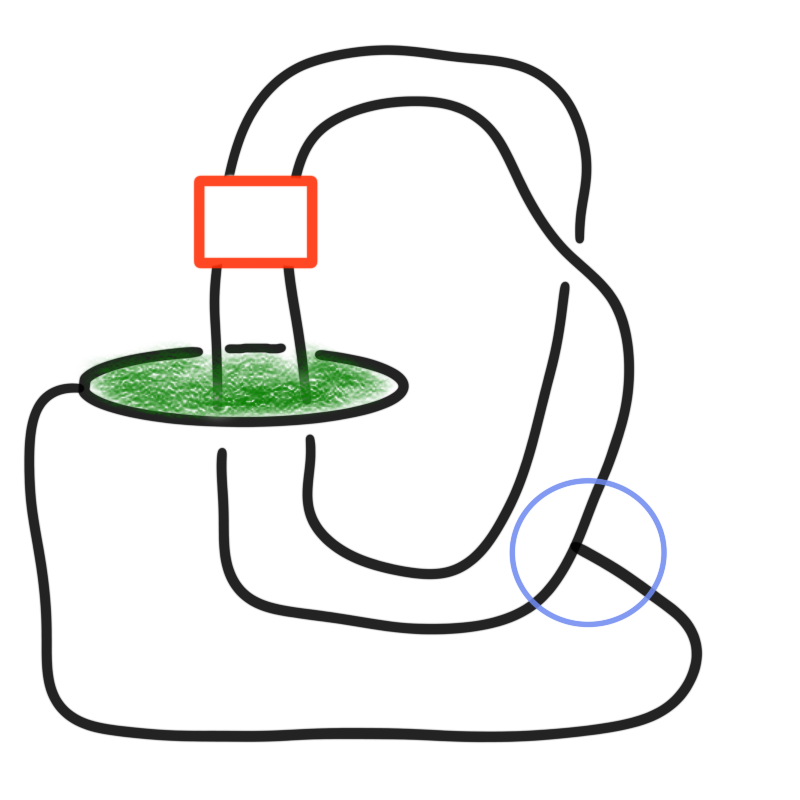}
			\put(30,70){$k$}
			\put(75,32){\footnotesize $B$} 	
			\put(19,49){\footnotesize $D$} 	
			\put(83.5,26){\tiny $x$} 	
			\put(66,20){\tiny $y$} 	
			\put(72,41){\tiny $z$} 	
		\end{overpic}
		\caption{Twisting disk $D$, $\Gamma^k$ and $B$.}
		\label{fig:Gamma_k}
	\end{subfigure}
	\begin{subfigure}{.45\linewidth}
		\centering
		\begin{overpic}[scale=.12,percent]{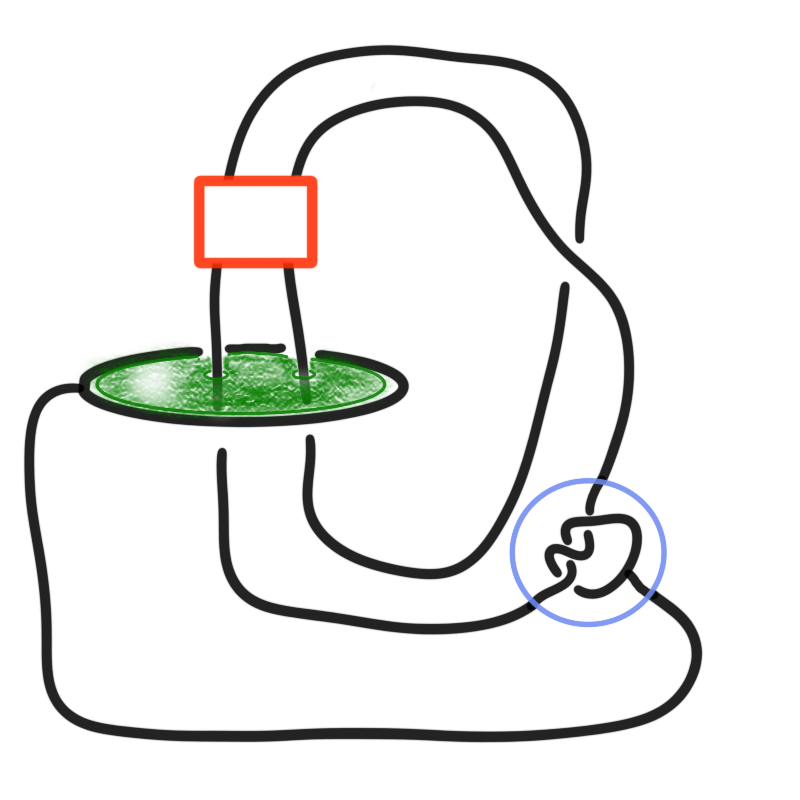}
			\put(30,70){$k$}
			\put(81,37){\footnotesize $(B,\alpha)$} 	
			\put(19,49){\footnotesize $S$} 	
		\end{overpic}
		\caption{The surface $S\subset \Compl{V^k_\alpha}$.}
		\label{fig:V_k_tau}
	\end{subfigure}
	\caption{}
\end{figure}

Suppose $\alpha$ is atoroidal and non-trivial.  
Then the exterior $\Compl {V^k_\alpha}$ admits an essential M\"obius band  $M^k_\ast \subset \Compl{V^k_\alpha}$ given by twisting $M$ in Fig.\ \ref{fig:mobius} $k$ times along $D$. The frontier of its regular neighborhood is an essential annulus $A^k_\ast$ with a slope of $k-\frac{1}{2}$. Also, the pair of pants $S$ in Fig.\ \ref{fig:fourone_tau} induces a pair of pants, namely $D\cap \Compl{V^k_\alpha}$; see Fig.\ \ref{fig:V_k_tau}; denote the pair of pants still by $S$.
Let $\Split {V^k_\alpha}$ be the $3$-manifold obtained by splitting $\Compl{V^k_\alpha}$ along $S$. Let $P^k,F^k\subset \partial V^k_\alpha$ be the components cut off by $\partial S$.
Then the $3$-manifold pair $(\Split {V^k_\alpha},P^k\cup F^k)$ is homeomorphic to 
$(\Split{V_\alpha},P\cup F)$. As a result, Lemma \ref{lm:rectangles_split_manifold} holds true for $V^k_\alpha$ as well, and Theorem \ref{lm:annuli} can be generalized to $V^k_\alpha$.

This allows to define an $D_\ast^k$-surface to be the pair of pants $G$
who has two boundary components parallel to the non-separating disk $D_\ast^k$ disjoint from $M^k_\ast$ and has the other boundary component non-meridional, generalizing Definition \ref{def:M_surface}.  
Since the $3$-manifold pairs $(\Split {V^k_\alpha},P^k\cup F^k)$ and $(\Split{V_\alpha},P\cup F)$
are homeomorphic, the proof of Lemma \ref{lm:unique_M_surface} can be generalized and shows that $S$ is the unique $D_\ast^k$-surface, which implies an analogs of Lemma \ref{lm:equivalence}: 
\begin{lemma}\label{lm:gen_equivalence}
Suppose $\alpha,\beta$ are atoroidal and not $\frac{1}{n}$-rational. Then 
\begin{enumerate}[label=(\roman*)]		\item\label{itm:gen_orientation_preserving} every homeomorphism $f:(\sphere,V^k_\alpha)\rightarrow (\sphere,V^k_\beta)$ is orientation-preserving;
\item\label{itm:gen_restriction} every equivalence between $V^k_\alpha,V^k_\beta$ induces an $x$-equivalence between $\alpha,\beta$. 
\end{enumerate} 
\end{lemma}
  
As a generalization of Theorem \ref{teo:classification}, we have the following.
\begin{theorem}
Suppose $\alpha,\beta$ are atoroidal.
Then $V^k_\alpha, V^k_\beta$ are equivalent if and only if $\alpha,\beta$ are $x$-equivalent.  
\end{theorem}
\begin{proof} 
The case where both $\alpha,\beta$ are not $\frac{1}{n}$-rational follows from Lemma \ref{lm:gen_equivalence}\ref{itm:gen_restriction}.
 
Suppose $\alpha$ and $\beta$ are $\frac{1}{n}$-rational and $\frac{1}{m}$-rational, respectively. 
If $n,m\neq \pm 3$, then the assertion follows by comparing the slopes of the two essential annuli in $\Compl{V^k_\alpha}$: $A_\ast^k$ has a slope of $k-\frac{1}{2}$ and the other, given by twisting $A_n$ in Fig.\ \ref{fig:Aprime} $k$ times, has a slope of $4k+\frac{n-4}{2}$.

If $n,m=\pm 3$, it follows from \cite[Theorem $1.1$]{LeeLee:12} since, if $\alpha$ is $\frac{1}{3}$-rational, then $V^k_\alpha$ is equivalent to the handlebody-knot $V_k$ in \cite{LeeLee:12}, while $\overline{V_{\overline{\alpha}}^k}$ is 
equivalent to $V_\alpha^{-k+1}$, equivalent to the handlebody-knot $V_{-k+1}$ in \cite{LeeLee:12}. 
\end{proof}

\begin{remark}
The handlebody-knot $\seventhirtynine$ (resp.\ $\sevensixty$) can be obtained by twisting $\sixtwelve$ (resp.\ $\sevenfiftynine$) along the disk $D$ once. 
Therefore the exteriors of $\sixtwelve,\seventhirtynine$ (resp.\ $\sevenfiftynine,\sevensixty$) are homeomorphic, so invariants depending only on the handlebody-knot exterior cannot differentiate them. 
\end{remark}


 

 

\end{document}